\documentclass[11pt]{amsart}

\usepackage{geometry, amsmath, amssymb, amsthm, mathrsfs, setspace, comment, amscd, latexsym, mathtools, bm}
\usepackage[all]{xy}                          
\usepackage{parskip}
\usepackage[abs]{overpic} 
\usepackage{pict2e}

\usepackage{xcolor,graphicx}

\usepackage{hyperref}
\hypersetup{
    colorlinks=true,
    citecolor=blue,
    linkcolor=blue,
    urlcolor=blue,
}

\makeatletter
\@addtoreset{equation}{section}

\makeatother

\newtheorem{theorem}{Theorem}[section]
\newtheorem{lemma}[theorem]{Lemma}
\newtheorem{corollary}[theorem]{Corollary}
\newtheorem{proposition}[theorem]{Proposition}
\theoremstyle{definition}
\newtheorem{definition}[theorem]{Definition}

\newtheorem{problem}[theorem]{Problem}

\newtheorem{remark}[theorem]{Remark}

\newcommand{\del}{\partial}

\newcommand{\Z}{\mathbb{Z}}
\newcommand{\R}{\mathbb{R}}
\newcommand{\C}{\mathbb{C}}
\newcommand{\CP}{\mathbb{CP}}
\newcommand{\M}{\mathcal{M}}
\newcommand{\D}{\mathbb{D}}

\DeclareMathOperator{\Crit}{Crit}
\DeclareMathOperator{\Critv}{Critv}
\DeclareMathOperator{\Symp}{Symp}

\DeclareMathOperator{\PD}{PD}
\DeclareMathOperator{\Int}{Int}

\renewcommand{\O}{\mathcal{O}}

\newcommand{\red}[1]{\textcolor{magenta}{#1}}

\title{A four-dimensional mapping class group relation}
\author[Takahiro Oba]{Takahiro Oba}
\address{Department of Mathematics, Kyoto University, Kyoto 606-8502, Japan}
\email{taka.oba@math.kyoto-u.ac.jp}
\address{Current address: Department of Mathematics, The University of Osaka, Toyonaka, Osaka 560-0043, Japan}
\email{taka.oba@math.sci.osaka-u.ac.jp}
\date{\today}

\begin{document}

\maketitle

%\tableofcontents

\begin{abstract}
In the symplectic mapping class group of a $4$-dimensional Weinstein domain, we give a relation between two products of (right-handed) Dehn twists via holomorphic curve techniques.
A key ingredient of the construction is a solution to the symplectic isotopy problem on symplectic submanifolds in del Pezzo surfaces. 
In the appendix, we provide an alternative proof of a relation between a fibered Dehn twist and a product of Dehn twists. %using Lefschetz--Bott fibrations.
\end{abstract}

\section{Introduction}

Let $(W,\omega)$ be a symplectic manifold possibly with boundary. 
A \textit{compactly supported symplectic automorphism} is a diffeomorphism $\varphi \colon W \rightarrow W$ such that $\varphi^{*}\omega=\omega$, and its support is compact and lies in the interior of $W$. 
Let $\Symp_c(W, \omega)$ denote the group of compactly supported symplectic automorphisms of $(W,\omega)$ equipped with the $C^{\infty}$-topology. 
This paper is concerned with relations in the \textit{symplectic mapping class group} $\pi_0(\Symp_c(W, \omega))$. 
In particular, we study the case where $(W,\omega)$ is a \textit{Weinstein domain} (see Definition \ref{def: Weinstein}). 

Relations in $\pi_0(\Symp_{c}(W, \omega))$ help us to understand the topology of Weinstein fillings of contact manifolds. 
Thanks to results of Akbulut and Ozbagci \cite{AO}, Loi and Piergallini \cite{LP} and Giroux and Pardon \cite{GP}, every Weinstein domain admits a Lefschetz fibration with regular fibers symplectomorphic to a codimension $2$ Weinstein domain, say $(W, \omega)$. 
Hence, Lefschetz fibrations relate $\pi_0(\Symp_c(W,\omega))$ to Weinstein domains via their monodromies. 
When $W$ is a surface, Moser's trick tells us that $\pi_{0}(\Symp_{c}(W, \omega))$ is isomorphic to the smooth mapping class group of $W$ (see e.g. \cite[Lemma 7.3.2]{Gei}). 
Since the latter group is well studied, Ozbagci and Stipsicz \cite{OS}, for example, handled relations in smooth mapping class groups of surfaces and constructed contact $3$-manifolds with infinitely many Weinstein fillings up to homotopy. 
% (see \cite{}, \cite{} for more results in this direction). 

When $\dim W \geq 4$, applying the above results on Lefschetz fibrations is not yet practical for constructing various Weinstein domains. 
Although symplectic mapping class groups of some $4$-dimensional Weinstein domains are understood  (e.g. \cite{SeiDehn}, \cite{Evans} and \cite{Wu}), still little is known about those groups of more general Weinstein domains and especially relations in them.

Relations in symplectic mapping class groups are often derived from fibration-like structures on symplectic manifolds.
Seidel \cite{SeiKnot} extracted a braid relation about two Dehn twists along Lagrangian spheres from Lefschetz fibrations. 
Lefschetz pencils and fibrations yield a relation between a fibered Dehn twist and a product of (right-handed) Dehn twists: \cite{Aur}, \cite{Seidel}, \cite{Go}, \cite{AA} (see also Theorem \ref{thm: boundary-interior}). 
This kind of relation involves many important ones in smooth mapping class groups of surfaces such as chain relations \cite{AA} and the lantern relation \cite{AS}. 
Note that any fibered Dehn twist defined on a surface is actually a Dehn twist. 
In another direction, Keating \cite{Keating} compared relations in higher-dimensional symplectic mapping class groups with $2$-dimensional ones in terms of Lefschetz fibrations. 

Reviewing constructions of $4$-dimensional Weinstein domains, relations between two products of \textit{different numbers} of Dehn twists such as the lantern relation play a crucial role in construction: e.g. \cite{EG}, \cite{EMVH}, \cite{BVH} and \cite{DKP}. 
The main theorem in this paper is to provide the first example of such relations in symplectic mapping class groups of $4$-dimensional Weinstein domains.

\begin{theorem}\label{thm: relation}
There exists a $4$-dimensional Weinstein domain $(W,\omega)$ and Lagrangian spheres $L_{1,1}, \cdots, L_{1,4}, L_{2,1}, \cdots, L_{2,6}$ such that 
the two products of (right-handed) Dehn twists along these Lagrangian spheres, $\tau_{L_{1,1}} \circ \cdots \circ \tau_{L_{1,4}}$ and $\tau_{L_{2,1}} \circ \cdots \circ \tau_{L_{2,6}}$, are symplectically isotopic; in other words, we have
$$[\tau_{L_{1,1}} \circ \cdots \circ \tau_{L_{1,4}}] = [\tau_{L_{2,1}} \circ \cdots \circ \tau_{L_{2,6}}] \in \pi_{0}(\Symp_{c}(W, \omega)).$$
\end{theorem}

We will construct the Weinstein domain $(W, \omega)$ and Lagrangian spheres $L_{i,j}$ in the theorem, employing fibration-like structures. 
Here is a sketch of the proof of Theorem \ref{thm: relation}:  
It is known that there are distinct two complex $3$-folds $X_1$ and $X_2$ containing del Pezzo surfaces of degree $6$ as ample divisors. 
Note that such a del Pezzo surface is diffeomorphic to $\CP^2 \# 3 \overline{\CP}^2$. 
For $i=1,2$, consider a Lefschetz pencil $f_i \colon X_i \dashrightarrow \CP^1$ defined by a linear system containing the above ample divisor. 
This gives a Lefschetz fibration $p_i \colon X_i \setminus \nu(f_i^{-1}(\infty)) \rightarrow D^2$, where $\nu(f_i^{-1}(\infty))$ is a tubular neighborhood of the regular fiber $f_i^{-1}(\infty)$ in $X_i$. 
Let $(W_i, \omega_i)$ denote a regular fiber of $p_i$ ($i=1,2$), which is a Weinstein domain. 
According to the aforementioned results \cite{Aur} and \cite{Go}, the global monodromy of $p_i$ is a fibered Dehn twist $\tau_{\del W_i}$ along $\del W_i$ that factors into the product of Dehn twists along Lagrangian spheres $L_{i,j}$: 
$$
	[\tau_{\del W_1}]=[\tau_{L_{1,1}} \circ \cdots \circ \tau_{L_{1,4}}] \quad \textrm{and} \quad [\tau_{\del W_2}]=[\tau_{L_{2,1}} \circ \cdots \circ \tau_{L_{2,6}}].
$$
The key step of this construction is to show that $(W_1,\omega_1)$ and $(W_2,\omega_2)$ are symplectomorphic; this is a consequence of the symplectic isotopy problem (see Problem \ref{prob: isotopy}). 
We solve it for anti-canonical divisors on del Pezzo surfaces via holomorphic curve techniques.
The symplectomorphism between $(W_1,\omega_1)$ and $(W_2,\omega_2)$ allows us to identify $\tau_{\del W_1}$ with $\tau_{\del W_2}$. 
Setting $(W,\omega) \coloneqq (W_1,\omega_1)$, we finally obtain the desired relation in $\pi_0(\Symp_c(W,\omega))$.  
Remark that some Lagrangian spheres in Theorem~\ref{thm: relation} might be the same. 
We also find that giving a symplectic mapping class relation is somewhat more difficult than giving a smooth one since our construction relies on the symplectic isotopy problem.

We would like to point out a result for the symplectic mapping class group of $\CP^2 \# 3\overline{\CP}^2$. 
Capping off the boundary $\del W$ of the Weinstein domain $(W, \omega)$ in Theorem~\ref{thm: relation} by a disc bundle over a $2$-torus, we obtain $\CP^2 \# 3\overline{\CP}^2$ with a certain symplectic structure $\omega'$ and a relation between two products of Dehn twists on its symplectic mapping class group canonically induced by Theorem \ref{thm: relation}. 
Li, Li and Wu \cite{LLW} showed that the symplectic mapping class group $\pi_0(\Symp_c(\CP^2 \# 3\overline{\CP}^2, \omega'))$ is a finite group generated by  Dehn twists. 
Moreover, it also follows from their argument that the square of any Dehn twist is symplectically isotopic to the identity. 
Hence, one could easily find a relation between two products of Dehn twists in $\pi_0(\Symp_c(\CP^2 \# 3\overline{\CP}^2, \omega'))$.
However, such a relation does not imply Theorem \ref{thm: relation} directly yet unless we understand the difference between $\pi_0(\Symp_c(\CP^2 \# 3\overline{\CP}^2, \omega'))$ and $\pi_0(\Symp_c(W, \omega))$.  

With the help of Lefschetz fibrations and open books, the author \cite{ObaFillings} constructed $(4n-1)$-dimensional contact manifolds admitting infinitely many Weinstein fillings, or equivalently Stein fillings up to homotopy. 
After the first version of this paper was posted on arXiv, intriguingly, Zhou \cite{Zhou_fillings} provided contact manifolds with infinitely many Weinstein fillings \textit{in every dimension}, up to homotopy, based on Lefschetz fibrations and spinal open books, a variant of open books. 
Nonetheless, unlike $3$-dimensional results such as those in \cite{BVH} and \cite{DKP}, it is still difficult to control the Euler characteristics of Weinstein fillings of contact manifolds of dimension $\geq 5$. 
However, by using an $h$-principle argument, Lazarev \cite{Laz} has constructed contact manifolds in any dimension with arbitrarily many finite Weinstein fillings, each having distinct Euler characteristics.
As an application of Theorem \ref{thm: relation}, we will partially recover this result by means of fibration-like structures (see Proposition \ref{prop: fillings}).

\subsection*{Organization of this paper}
Section \ref{section: Dehn twists and fibered twists} deals with Dehn twists, fibered Dehn twists and their relation. 
We begin by giving basic definitions related to symplectic manifolds in Section \ref{section: basics} and then review Dehn twists in Section \ref{section: Dehn twists} and Lefschetz fibrations in Section \ref{section: LFs}. 
After this, we present a way to obtain a Lefschetz fibration from a given Lefschetz pencil in Section \ref{section: LP to LF}. 
Applying this, Section \ref{section: boundary-interior} explains a symplectic mapping class relation between a fibered Dehn twist and a product of Dehn twists. 
Section \ref{section: isotopy} is devoted to the symplectic isotopy problem for del Pezzo surfaces. 
Collecting results on holomorphic curves in Section \ref{section: holomorphic curves}, we give a solution to the problem in Section \ref{section: solution}. 
Section \ref{section: MCG} begins with a review of how del Pezzo surfaces appear as ample divisors on complex $3$-folds in Section \ref{section: del Pezzo} and then proves Theorem \ref{thm: relation} and its application to Weinstein fillings of contact $5$-manifolds. 
Finally, we conclude this paper by showing a relation between a fibered Dehn twist and a product of Dehn twists, restricting to the complex projective case in Appendix \ref{appendix: boundary-interior}.

\subsection*{Conventions and Notations}
Let $p \colon E \rightarrow X$ be a vector bundle. 
We often call the total space $E$ the \textit{vector bundle}. 
Also the \textit{zero-section} of $p$ often means the image of the zero-section and denote it by $X$ with slight abuse of notation.
For a submanifold $Y \subset M$, we denote the pull-back of a differential form $\omega$ on $X$ to $Y$ under the inclusion by $\omega|_{Y}$. 

We often consider complex manifolds, especially K\"{a}hler manifolds, and their complex submanifolds in this paper. An $n$-dimensional manifolds simply means a manifold of \textit{real} dimension $n$. When indicating the complex dimension of a complex manifold, it is said to be a complex $n$-dimensional manifold or complex manifold of complex dimension $n$. We also denote real and complex dimension by $\dim_{\R}$ and $\dim_{\C}$, respectively. 
A \textit{complex surface} (resp. \textit{curve}) means a complex manifold of complex dimension two (resp. one). In addition, real/complex codimensions follow the same manner as real/complex dimensions. 

\subsection*{Acknowledgements} 
The author would like to thank Akihiro Kanemitsu, Kaoru Ono and Toru Yoshiyasu for helpful discussions. 
He was supported by Japan Society for the Promotion of Science KAKENHI Grant Numbers 20K22306, 22K13913 and 24H00182.

%%%%%%%%%%%%%%%%%%%%%%%%
%%%%%%%%%%%%%%%%%%%%%%%%

\section{Dehn twists and fibered Dehn twists}\label{section: Dehn twists and fibered twists}

\subsection{Symplectic preliminaries}\label{section: basics}

We first recall some basics definitions related to symplectic manifolds. 
Let $(W, \omega)$ be a compact symplectic manifold with boundary $\del W$. 
The symplectic manifold $(W,\omega)$ is called a \textit{convex symplectic domain} if there exists a vector field $X$ defined in a collar neighborhood of $\del W$, called a \textit{Liouville vector field}, such that $\mathcal{L}_X \omega=\omega$ and $X$ points outward along the boundary $\del W$. 
If $X$ is defined on the whole $W$, then $(W,\omega)$ is called a \textit{Liouville domain}. 
Note that for the Liouville vector field $X$ on a convex symplectic domain, the $1$-form $i_X \omega$ satisfies $\omega=d i_X \omega$ and gives a contact form on $\del W$ (see \cite[Lemma/Definition~1.4.5]{Gei} for the proof). 
The \textit{completion} of a convex symplectic domain $(W, \omega)$ with Liouville vector field $X$ is the open symplectic manifold $(\widehat{W}, \widehat{\omega})$ defined by 
$$
	(\widehat{W}, \widehat{\omega}) = (W, \omega) \cup ([0,\infty) \times \del W, d(e^t((i_X\omega)|_{\del W})) ),
$$
where $t$ is the coordinate on $[0,\infty)$, and $W$ and $[0,\infty) \times \del W$ are glued together by identifying $\del W$ with $\{0\} \times \del W$.

\begin{definition}\label{def: Weinstein}
A \textit{Weinstein domain} is a Liouville domain $(W,\omega)$ such that the Liouville vector field $X$ is gradient-like for a Morse function $f \colon W \rightarrow \R$ with no critical points on $\del W$ and $f^{-1}(\max f)=\del W$, where $\max f$ denotes the maximum of the function $f$. 
\end{definition}

The boundary $\del W$ of a Weinstein domain $(W, \omega)$ of $\dim_{\R}W=2n$ ($n \geq 1$) carries a contact structure defined by $(i_X\omega)|_{\del W}$. 
Here a \textit{contact structure} on an oriented manifold $M$ of $\dim_{\R}M=2n-1$ is a hyperplane distribution written as the kernel of a $1$-form $\alpha$ on $M$ satisfying that $\alpha \wedge (d\alpha)^{n-1}$ is a positive volume form on $M$. 
Note that $\del W$ is canonically oriented as $W$ is oriented by $\omega^n$. 
A \textit{Weinstein filling} of a contact manifold $(M,\xi)$ is a Weinstein domain $(W,\omega)$ satisfying $\del W=M$ as oriented manifold and $\ker((i_X\omega)|_{\del W})=\xi$. 

Let $h \colon V \rightarrow \R$ be a smooth function on a symplectic manifold $(W,\omega)$. 
We define the \textit{Hamiltonian vector field} associated to $h$ as a vector field $X_h$ such that $dh=-\iota_{X_{h}}\omega$. 

We will deal with a special type of complex manifolds, called K\"{a}hler manifolds. 
Let $W$ be a complex manifold with the integrable almost complex structure $J$ induced by the complex structure.
A symplectic form $\omega$ on $W$ is said to be \textit{K\"ahler} if $J$ is compatible with $\omega$, i.e., $\omega(v,Jv)>0$ for any non-zero vector $v \in TW$ and $\omega_x(Jv_1, Jv_2) =\omega_x(v_1,v_2)$ for all $x \in W$ and all $v_1, v_2 \in T_{x}W$. 
Then, the triple $(W,J,\omega)$ is called a \textit{K\"{a}hler manifold}. 

To conclude this subsection, we define the notion of holomorphicity of a map and its related notion, which will be used in the definition of a Lefschetz fibration for instance. 
Let $(E,J)$ and $(S,j)$ be almost complex manifolds. 
A differentiable map $f \colon E \rightarrow S$ is said to be \textit{$(J,j)$-holomorphic} if $df \circ J=j \circ df$. 
Suppose that $J$ is integrable, $S=\C$ and $x \in E$ is a critical point of a $(J,j)$-holomorphic $C^2$-map $f$. 
The \textit{complex Hessian} of $f$ at $x$ is represented by $(\del^2 (\psi \circ f \circ \varphi^{-1})/\del z_i \del z_j)_{1 \leq i,j \leq n}$ for complex coordinate maps $\varphi=(z_1, \ldots, z_n)$ around $x \in E$ and $\psi$ around $f(x) \in \C$. 
Note that this matrix itself is dependent on the choice of coordinates; however, its nondegeneracy is independent of the choice of them. 

\subsection{Dehn twists}\label{section: Dehn twists}

Here we briefly review the definition of Dehn twist. 
The reader is referred to \cite[(16c)]{SeiBook} and \cite[(1a)]{Sei} for more details. 

First let us recall a model Dehn twist on the cotangent bundle $T^{*}S^{n}$ to the sphere $S^{n}$. 
We regard $T^{*}S^{n}$ as the space defined by 
$$
	\{(q, p) \in \R^{n+1} \times \R^{n+1} \mid q \cdot q =1, q \cdot p=0\}, 
$$
where $\cdot$ denotes the standard inner product of the Euclidean space $\R^{n+1}$. 
Define $\lambda_{\mathrm{can}} \coloneqq pdq|_{T^{*}S^n}$.
Consider the normalized geodesic flow $\sigma_{t}$ with $t \in \R$ on $T^{*}S^{n} \setminus S^n$ given by  
$$
	\sigma_{t}(q,p)=(\cos(t)q+|p|^{-1}\sin(t)p, -|p|\sin(t)q+\cos(t)p). 
$$
A \textit{model Dehn twist} is a diffeomorphism of $T^{*}S^{n}$ defined by 
$$
	\tau(q,p) \coloneqq 
	\begin{cases}
	\sigma_{f(|p|)}(q,p) & \textrm{if\ } p \neq 0,\\
	(-q,0) & \textrm{if\ } p=0, 
	\end{cases}
$$
where $f \colon [0,\infty) \rightarrow \R$ is a smooth function such that $f(r)$ equals $\pi$ near $r=0$ and $0$ for $r \gg 0$, and $f'(r) \leq 0$ for any $r \in [0,\infty)$. 
%Notice that for a point on the zero-section, i.e., $p=0$, $\tau$ agrees with the antipodal map of $S^n$. 
The diffeomorphism $\tau$ is an exact symplectomorphism of $(T^{*}S^{n}, d\lambda_{\mathrm{can}})$ with compact support. 
Choosing a function $f$ whose support is sufficiently small, we may assume that the support of $\tau$ lies in a small neighborhood of the zero-section of $T^{*}S^n$. 

Now suppose that we have a Lagrangian sphere $L$ in a $2n$-dimensional symplectic manifold $(W,\omega)$, i.e., a submanifold of $(W,\omega)$ satisfying that $L$ is diffeomorphic to $S^n$ and $\omega|_L \equiv 0$. 
Equip $L$ with a diffeomorphism $v \colon S^{n} \rightarrow L$, called a faming of $L$. 
With the faming $v$, the Weinstein Lagrangian tubular neighborhood theorem gives a symplectic embedding $\iota \colon (D_{\epsilon}^{*}S^{n}, d\lambda_{\mathrm{can}}) \rightarrow (W, \omega)$ satisfying that $\iota|_{S^{n}}=v$, where $D_{\epsilon}^{*}S^{n}$ is the disc cotangent bundle to $S^n$ of radius $\epsilon$. 
Define a \textit{(right-handed) Dehn twist} $\tau_{L}$ along $L$ by 
$$
	\tau_{L}(x) = \begin{cases}
	(\iota \circ \tau \circ \iota^{-1})(x) & \textrm{if \ } x \in \mathrm{Im}(\iota), \\
	x & \textrm{otherwise},
	\end{cases}
$$
where we choose a model twist $\tau$ to be supported in $D_{\epsilon}^{*}S^{n}$. 
By definition, $\tau_L$ is a compactly supported symplectic automorphism of $(W, \omega)$. 
In fact, it is independent of the choice not only of $\iota$ and but also of $v$ when $n \leq 2$ up to Hamiltonian isotopy \cite[Remark 16.1]{SeiBook}.  

\subsection{Lefschetz fibrations}\label{section: LFs}

%\subsubsection{Definition of Lefschetz fibration and its related notions}

Let $E$ be an open manifold of dimension $2n$ ($n \geq 1$) and $\Omega$ a closed $2$-form on $E$. 

\begin{definition}\label{def: Lefschetz}
A smooth map $\pi \colon (E, \Omega) \rightarrow \C$ is called a \textit{Lefschetz fibration} if 
the set of critical points of $\pi$, denoted by $\Crit(\pi)$, is a finite subset in $E$; 
there exists an $\Omega$-compatible almost complex structure $J$ defined in a neighborhood $\mathcal{U}$ of 
$\Crit(\pi)$ in $E$; there is a positively oriented complex structure $j$ defined in a neighborhood $\mathcal{V}$ of $\Critv(\pi) \coloneqq \pi(\Crit(\pi))$ in $\C$, which satisfy the following conditions: 
\begin{enumerate}
\item\label{cond: holomorphic} The map $\pi|_{\mathcal{U}}: \mathcal{U} \rightarrow \mathcal{V}$ is $(J,j)$-holomorphic. 
\item\label{cond: symplectic fibers} The $2$-form $\Omega$ is nondegenerate on $\ker_{x}(D\pi)$ for any point $x \in E$. 
\item For all $z \in \C \setminus \Critv(\pi)$, the fibers $\pi^{-1}(z)$ are symplectomorphic to the completion $(\widehat{F}, \widehat{\omega}_F)$ of a convex symplectic domain $(F, \omega_F)$ of $\dim_{\R}F=2n-2$, that is, 
$$
	(\widehat{F}, \widehat{\omega}_F) = (F, \omega_F) \cup ([0,\infty) \times \del F, d(e^t((i_X\omega_F)|_{\del F}))), 
$$
where $X$ is a Liouville vector field defined near the boundary $\del F$. 
\item\label{cond: near critical points} Near $\Crit(\pi)$, the almost complex structure $J$ is integrable and the $2$-form $\Omega$ is K\"ahler. 
\item\label{cond: Hessian} The complex Hessian $D^2_{x} \pi$ of $\pi$ is nondegenerate at any critical point $x \in \Crit(\pi)$. 
\end{enumerate}
\end{definition}

A fiber of a Lefschetz fibration $\pi \colon (E, \Omega) \rightarrow \C$ is said to be \textit{singular} if it is over a critical value of $\pi$; otherwise it is said to be \textit{regular}. 

\begin{remark}\label{rem: horizontal}
Let  $\pi \colon (E, \Omega) \rightarrow \C$ be a Lefschetz fibration whose regular fibers are symplectomorphic to the completion of a convex symplectic domain $(F, \omega_{F})$ with Liouville vector field $X$ near the boundary $\del F$ as in Definition \ref{def: Lefschetz}. 
We often further require the fibration $\pi$ to satisfy the following condition: 
\begin{itemize}
\item \textit{Horizontal triviality}. There is a subset $E'$ of $E$ such that the restriction $\pi|_{\red{E'}}$ is isomorphic to the trivial fibration 
$$
	\C \times [0, \infty) \times \del F \rightarrow \C,
$$ and $\Omega$ is identified with $\omega_{\C} + d(e^t((i_{X}\omega_{F})|_{\del F}))$ by the isomorphism, where $\omega_{\C}$ is the pull-back of a symplectic form on $\C$.
\end{itemize}
A Lefschetz fibration satisfying this condition is referred to as a \textit{Lefschetz fibration with horizontal triviality}. 
\end{remark}

\begin{remark}\label{rem: compact fibers}
A Lefschetz fibration $\pi \colon (E,\Omega) \rightarrow \C$ with horizontal triviality provides a Lefschetz fibration over $\C$ with compact fibers. 
Indeed, for the subset $E'$ of $E$ given in the above remark, take a subset $E'' \subset E'$ on which $\pi$ is identified with 
$$
	\C \times [\epsilon, \infty) \times \del F \rightarrow \C,
$$
where $\epsilon>0$ is a constant. 
Removing $E''$ from $E$, we have the desired Lefschetz fibration $\pi|_{\overline{E \setminus E''}} \colon (\overline{E \setminus E''}, \Omega|_{\overline{E \setminus E''}}) \rightarrow \C$.
By construction, it is straightforward that on a collar neighborhood of the boundary $\del (\overline{E \setminus E''})$, the map $\pi|_{\overline{E \setminus E''}}$ is regarded as the trivial fibration $\C \times [0, \epsilon] \times \del F \rightarrow \C$. 
We say that the Lefschetz fibration $\pi|_{\overline{E \setminus E''}} \colon (\overline{E \setminus E''}, \Omega|_{\overline{E \setminus E''}}) \rightarrow \C$ meets the \textit{horizontal triviality} as in the case of Lefschetz fibrations with non-compact fibers. 
\end{remark}

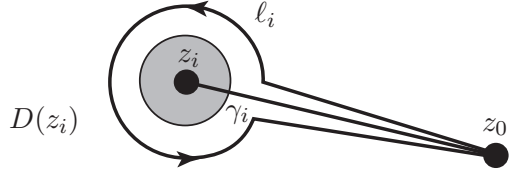
\begin{figure}[t]
	\centering
	\begin{overpic}[width=150pt,clip]{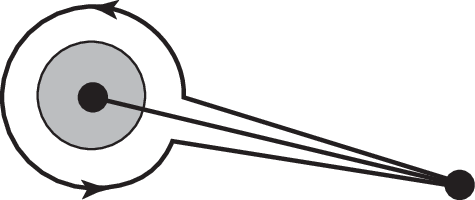}
	 \linethickness{1pt}
	\put(140,14){$z_0$}  
 	\put(26,40){$z_{i}$}
	\put(-10,20) {\vector(2,1){31}}
  	\put(-37,15){$D(z_i)$}
	\put(55,55){$\ell_i$}
	\put(44,19){$\gamma_i$}
	\end{overpic}
	\caption{Vanishing path $\gamma_i$ and loop $\ell_i$. }
	\label{fig: loop}
\end{figure}

Set $\Critv(\pi)=\{z_1, \ldots, z_k\}$ for a Lefschetz fibration $\pi \colon (E, \Omega) \rightarrow \C$ with horizontal triviality. 
Assume that $\pi|_{\Crit(\pi)} \colon \Crit(\pi) \rightarrow \Critv(\pi)$ is injective. 
According to the condition (\ref{cond: symplectic fibers}) in Definition \ref{def: Lefschetz}, the symplectic connection is defined away from the singular fibers and determines the parallel transport map along a path in $\C \setminus \Critv(\pi)$. 
Fix a basepoint $z_0 \in \C \setminus \Critv(\pi)$ with $|z|$ sufficiently large and set $F_{0}=\pi^{-1}(z_0)$ and $\omega_{F_{0}}=\Omega|_{F_{0}}$. 
A \textit{vanishing path} for $z_i \in \Critv(\pi)$ is an embedded path  $\gamma_i \colon [0,1] \rightarrow \C$ such that $\gamma_i(0)=z_0$, $\gamma_i(1)=z_i$ and $\gamma_i^{-1}(\Critv(\pi))=\{1\}$. 
One can associate to a vanishing path $\gamma_i$ for $z_i$ a framed Lagrangian sphere $V(\gamma_i)$ in $(F_{0}, \omega_{F_{0}})$, called the \textit{vanishing cycle} of $\gamma_i$. 
Take a small disc $D(z_i)$ centered at $z_i \in \Critv(\pi)$ and orient $\del D(z_i)$ counterclockwise. 
Let $\ell_i \colon [0,1] \rightarrow \C \setminus \Critv(\pi)$ be a loop obtained by welding a vanishing path $\gamma_i$ and $\del D(z_i)$ together (see Figure \ref{fig: loop}). %as in Figure \ref{fig} a \textit{vanishing loop} with respect to $\gamma_i$. 
An ordered collection $(\gamma_1, \ldots, \gamma_k)$ of vanishing paths is called a \textit{distinguished basis} of vanishing paths if it satisfies the following conditions: 
\begin{itemize}
	\item Any two distinct paths $\gamma_i$ and $\gamma_j$ intersect only at the basepoint $z_0$. 
	\item The concatenation $\ell_1 \cdots \ell_k$ of loops $\ell_i$ obtained by $\gamma_i$ in the above manner is homotopic to the circle $\ell_0(t)=z_0e^{2\pi i t}$ relatively to the basepoint $z_0$.
\end{itemize}
Let $\ell$ be a loop based at $z_0$ homotopic to $\ell_0$. 
As is known, for a chosen distinguished basis $(\gamma_1, \cdots, \gamma_k)$, the monodromy along $\ell$ is given by 
$$
	\tau_{V(\gamma_k)} \circ \cdots \circ \tau_{V(\gamma_1)} \in \Symp_{c}(F_{0}, \omega_{F_0})
$$ 
up to symplectic isotopy. 
We refer to this monodromy as the \textit{global monodromy} of $\pi$. 
Note that the global monodromy itself does not depend on the choice of distinguished basis $(\gamma_1, \ldots, \gamma_k)$ up to Hamiltonian isotopy, %\cite[Theorem 6.4.1]{MS_Intro}
whereas its factorization into Dehn twists does depend on its choice.

\subsection{From Lefschetz pencils to Lefschetz fibrations}\label{section: LP to LF}

In this subsection, we will explain how to obtain a Lefschetz fibration from a given Lefschetz pencil. 
%Before that, we first describe a symplectic form on a neighborhood of the base locus of a Lefschetz pencil. 

\subsubsection{Symplectic forms on Whitney sums}\label{section: Whitney sum}

To begin with, we describe a symplectic form on the Whitney sums of two line bundles, which will serve as a symplectic form on a neighborhood of the base locus of a Lefschetz pencil. 
The following argument is inspired by the proof of \cite[Theorem 2.3]{Go}. %\note{LP in highr dimensions}\note{One can skip the discussion and check the proposition.}

Let $(B, \omega_{B})$ be a closed integral symplectic manifold of arbitrary dimension, i.e., $[\omega_B] \in H^{2}(B; \R)$ has an integral lift and $p \colon L \rightarrow B$ a Hermitian complex line bundle with $c_1(L)=[\omega_{B}]$. 
Consider the Whitney sum $ \pi \colon  L^{\oplus 2}=L \oplus L  \rightarrow B$ and projectivization $f \colon L^{\oplus 2} \setminus B \rightarrow \CP^{1}$, that is, the map locally written as $f(b, x_1, x_2)=[x_1:x_2]$ for $(b, x_1,x_2) \in U \times \C \times \C$, where $U$ is an open set of $B$. 
Let $P_{\delta} \subset L^{\oplus 2}$ denote the sphere bundle of radius $\delta$ for the Hermitian metric associated to $L ^{\oplus 2}$. 
Restricting the map $(\pi, f) \colon L^{\oplus 2} \setminus B \rightarrow B \times \CP^{1}$ to $P_{\delta}$ gives a $U(1)$-bundle structure to $P_{\delta}$. 
One can identify $L^{\oplus 2} \setminus B$ with the complement $(P_{\delta} \times_{U(1)} \C) \setminus (B \times \CP^1)$ of the zero-section in the associated bundle by 
\begin{align}\label{eqn: identification}
	L^{\oplus 2} \setminus B \rightarrow (P_{\delta} \times_{U(1)} \C) \setminus (B \times \CP^1), \quad z=(z_0,z_1) \mapsto \left[{\delta z}/{|z|}, |z|\right], \end{align}
where $|z|$ denotes the norm of $z$ for the Hermitian metric. 
Take a connection $1$-form $\alpha$ on $P_{\delta}$ such that $d\alpha=-2\pi(\pi^{*}\omega_{B}-f^{*}\omega_{\textrm{FS}})$, where $\omega_{\textrm{FS}}$ is the Fubini--Study form on $\CP^1$. 
Now we define the $2$-form $\widetilde{\Omega}$ on $P_{\delta} \times \C$ by 
\begin{align*}
	\widetilde{\Omega}  \coloneqq \pi^{*}\omega_{B}-\frac{1}{2\pi}d(r^2 d\theta)+\frac{1}{2\pi}d(r^2\alpha),
\end{align*}
where $(r,\theta)$ are the polar coordinates on $\C$. 
Since this is $U(1)$-invariant and horizontal, it descends to a $2$-form not only on $P_{\delta} \times_{U(1)} \C$ but also on $L^{\oplus 2}$ under the identification (\ref{eqn: identification}).
%The $U(1)$-action is generated by $X_{\alpha}+\del_{\theta}$. Having the $r^2$ terms allows us to extend $\tilde{\Omega}$ to $L^{\oplus 2}$. 
We claim that the resulting $2$-form $\Omega$ on $L^{\oplus 2}$ is a symplectic form in a small neighborhood of the zero-section. 
As $\Omega$ is closed, it suffices to show the nondegeneracy of $\Omega$.  
In view of the choice of $\alpha$, we find that $\widetilde{\Omega}$ can be expressed as 
\begin{align}
	\widetilde{\Omega} & = \frac{1}{2\pi}d(r^2-1)( \alpha-d\theta )+f^{*}\omega_{\mathrm{FS}} \label{eqn: Omega} \\
	%& = (1-r^2)\pi^{*}\omega_{B}+(r^2-1)\frac{1}{2\pi}d(r^2 \theta)+\frac{1}{2\pi}d(r^2\alpha) \\
	 &= \frac{1}{2\pi}dr^2 \wedge ( \alpha-d\theta )+(1-r^2)\pi^{*}\omega_{B} +r^2f^{*}\omega_{\mathrm{FS}}. \label{eqn: Omega 2}
\end{align}
%The vector field $X_{\alpha}+\del_{\theta}$, where $X_{\alpha}$ is a generator of $U(1)$-action on $P_{\delta}$, spans the kernel of $\Omega$. 
%Since this vector field vanishes on $P_{\delta} \times_{U(1)} \C$ and $L^{\oplus 2}$, $\Omega$ is nondegenerate \red{as long as $r<1$}. 
Observe that $\Omega$ coincides with the standard symplectic form on the fibers of $\pi$ up to a constant factor. 
To see a horizontal direction, define the distribution $H$ of $TL^{\oplus 2}$ to be $TB$ on the zero-section and the $\alpha$-horizontal lifts of $T(B \times \{ \mathrm{pt} \}) \subset T(B \times \CP^{1})$ to $P_{\delta}$ for each $\delta>0$. 
%Here remark that $TB=T(B \times \{ pt\})$, and hence $H$ is tangent to the fibers of $f$.
Since $H$ is tangent to $P_{\delta}$ and the fibers of $f$, $H$ is the $\Omega$-complement of $\ker(D\pi)$. 
By (\ref{eqn: Omega 2}), $\Omega|_{H}=(1-r^2)\pi^{*}\omega_{B}$, and it extends smoothly to $B$. 
This concludes that $\Omega$ extends smoothly to the whole of $L^{\oplus 2}$ with $\Omega|_{B}=\pi^{*}\omega_{B}$. 
Hence, $\Omega$ is a nondegenerate $2$-form on a small neighborhood of the zero-section. 

Let $L^{\oplus 2}(r_0)$ denote the associated disc bundle to $L^{\oplus 2}$ of radius $r_0>0$ and set 
$$
N_{0}(r_0) \coloneqq L^{\oplus 2}(r_0) \setminus \{(z_0,z_1) \in L^{\oplus 2}(r_0) \mid z_0=0\}, 
$$
$$
N_{0}(r_0,r_1) \coloneqq N_{0}(r_{1}) \setminus \Int N_{0}(r_0) 
$$
for $0<r_0<r_1<1$.
We will see below that there is a Liouville vector field on $N_{0}(r_0,r_1)$ pointing outward along $\del N_{0}(r_0)$. 
First with (\ref{eqn: identification}), we regard $N_{0}(r_0, r_1)$ as $P_{\delta,0} \times_{U(1)} \mathbb{A}(r_0,r_1)$, where $P_{\delta,0}=P_{\delta} \setminus \{(z_0,z_1) \in  P_{\delta} \mid z_0=0 \}$ and $\mathbb{A}(r_0, r_1)=\{ z \in \C \mid r_{0} \leq |z| \leq r_1\}$.
Consider the function $h \colon P_{\delta,0} \times_{U(1)} \mathbb{A}(r_0,r_1) \rightarrow \R$ given by 
$$
h([z_0, z_1,re^{i\theta}])=r^2.
$$ 
Its Hamiltonian vector field $X_h$ on $P_{\delta,0} \times_{U(1)} \mathbb{A}(r_0,r_1)$ is $-2\pi\del_{\theta}$. 
%Note that this agrees with $2\pi R_{\alpha}$ on the quotient $P_{\delta,0} \times_{U(1)} \mathbb{A}(r_0,r_1)$, where $\R_{\alpha}$ is the Reeb vector field for $\alpha$. 
Since the image of $P_{\delta,0} \times_{U(1)} \mathbb{A}(r_0,r_1)$ by $f$ is $\C \subset \C \cup \{\infty\}= \CP^1$, the pull-back $f^{*}\omega_{\mathrm{FS}}$ restricts to $f^{*}\left(-\frac{1}{4\pi}dd^{\C} \log (|z|^2+1) \right)$ on $P_{\delta,0} \times_{U(1)} \mathbb{A}(r_0,r_1)$, where $z$ is the standard complex coordinate of $\C$ and $d^{\C} \coloneqq d \circ J_{\mathrm{st}}$ for the standard complex structure $J_{\mathrm{st}}$ on $\C$.
By (\ref{eqn: Omega}), $\Omega$ is exact on $P_{\delta,0} \times_{U(1)} \mathbb{A}(r_0,r_1)$. 
Hence, one can take the Liouville vector field $V$ on this region that is $\Omega$-dual to 
$$
	\lambda \coloneqq (r^2-1)(\alpha-d\theta)/2\pi-\frac{1}{4\pi}f^{*}d^{\C} \log (|z|^2+1).
$$ 
Now we have 
\begin{align*}
	dh(V) =  -\Omega(X_h, V) = \lambda(X_h) = r^2-1<0, 
\end{align*}
where $X_h$ is the Hamiltonian vector field associated to $h$.
This shows that $V$ points inward along each level set of $h$. 
In other words, $V$ points outward along $\del N_{0}(r_0)$ when it is considered as a part of the boundary of $N_{0}(r_0, r_1)$. %(see Figure \ref{fig: fiber}). 

Moreover, $\Omega$ matches the exact symplectic form $d(r^2-1)(\alpha-d\theta)/2\pi$ on each fiber of $f$, for which a Liouville vector field is given by 
\begin{align}\label{eqn: Liouville}
	\frac{r^2-1}{2r}\del_{r}. 
\end{align}
This makes the fibers of $f|_{N_{0}(r_0, r_1)}$ convex along the boundary component  lying in $\del N(r_0)$.

We summarize the above discussion into the following: 

\begin{proposition}
Let $(B, \omega_B)$ be a closed integral symplectic manifold of arbitrary dimension and $L \rightarrow B$ a Hermitian line bundle with $c_1(L)=[\omega_B]$. 
Also let $f \colon L^{\oplus 2} \setminus B \rightarrow \CP^1$ denote projectivization on each fiber. 
Then there exists a symplectic form $\Omega$ on a closed tubular neighborhood $\mathcal{U}$ of the zero-section $B$ of $L^{\oplus 2}$ such that the following holds: 
\begin{itemize}
\item On the zero-section, $\Omega$ agrees with $\omega_{B}$. 
\item For any disc bundle $L^{\oplus 2}(r)$ associated to $L^{\oplus 2}$ with $L^{\oplus 2}(r) \subset \mathcal{U}$, a Liouville vector field for $\Omega$ defined on $\mathcal{U}_{0} \coloneqq \mathcal{U} \setminus( \Int L^{\oplus 2}(r) \cup f^{-1}([0:1]))$ points outward along its boundary component lying in $\del L^{\oplus 2}(r)$. 
\item A Liouville vector field on each fiber of $f|_{\mathcal{U}_{0}}$ for the restriction of $\Omega$ to the fiber points outward along its boundary component lying in $\del L^{\oplus 2}(r)$.
\end{itemize}
\end{proposition}

\subsubsection{Lefschetz pencils and fibrations}
Let $(X, \omega)$ be a closed integral symplectic manifold of $\dim_{\R}X =2n \geq 4$ and $B \subset (X, \omega)$ a symplectic submanifold of $\dim_{\R}B=2n-4$.

\begin{definition}\label{def: LP}
A smooth map $f \colon X \setminus B \rightarrow \CP^{1}$, also denoted by $f \colon X \dashrightarrow \CP^1$, is called a \textit{Lefschetz pencil} on $(X, \omega)$ if it satisfies the following conditions: 
\begin{enumerate}
%\item For any $y \in \CP^1$, the closure of every fiber $F_{\red{y}} \coloneqq f^{-1}(\red{y})$ in $X$ agrees with $f^{-1}(\red{y}) \cup B$; 
\item\label{condition: nbhd} There exists a tubular neighborhood $\nu_{X}(B)$ of $B$ in $X$ such that 
\begin{itemize}
\item $\nu_{X}(B)$ is symplectomorphic to a neighborhood of the zero-section of $N \oplus N \rightarrow B$, where $N \rightarrow B$ is a complex line bundle over $B$ with $c_1(N)=[\omega|_B]$ and the Whitney sum is equipped with a symplectic form $\omega_{N \oplus N}$ as in Section \ref{section: Whitney sum}; 
\item $f|_{\nu_{X}(B)\setminus B}$ is identified with projectivization $(N \oplus N) \setminus B \rightarrow \CP^{1}$ by the above symplectomorphism.
\end{itemize}
\item\label{condition: LF} There exist almost complex structures $J$ and $j$ defined in neighborhoods of $\Crit(f)$ and $\Critv(f)$, respectively, which meet the conditions (\ref{cond: holomorphic}), (\ref{cond: near critical points}) and (\ref{cond: Hessian}) in Definition \ref{def: Lefschetz}. 
\item\label{condition: LP symplectic} The symplectic form $\omega$ on $X$ is nondegenerate on the closure $\overline{F}_{z}$ of every fiber $F_{z} \coloneqq f^{-1}(z)$ in $X$ except for the critical points of $f$. 
\end{enumerate}
\end{definition}

The submanifold $B$ is called the \textit{base locus} of a Lefschetz pencil $f \colon X \setminus B \rightarrow \CP^1$. 

According to the first condition in Definition \ref{def: LP}, the closure $\overline{F}_{z}$ of every regular fiber $F_z$ is a (real) $(2n-2)$-dimensional embedded submanifold of $X$. 
We also find that the complex line bundle $N \rightarrow B$ is isomorphic to the normal bundle $N_{B/\overline{F}_z} \rightarrow B$ of $B$ in $\overline{F}_{z}$. 

We shall see below that a Lefschetz pencil $f \colon X \setminus B \rightarrow \CP^1$ on $(X, \omega)$ induces a Lefschetz fibration on the complement of a regular fiber in $X$ after perturbing the symplectic structure $\omega$.  

\begin{proposition}\label{prop: LP to LF}
Let $f \colon X \setminus B \rightarrow \CP^1$ be a Lefschetz pencil on a closed integral symplectic manifold $(X, \omega)$. 
Suppose that $\infty \in \C \cup \{\infty\}=\CP^1$ is a regular value of $f$. 
Then, for some (open) tubular neighborhood $\nu_{X}(B)$ of $B$, there exists a closed $2$-form $\omega'$ on $E \coloneqq X \setminus ( \overline{F}_{\infty}  \cup \nu_X(B))$ such that the following holds: 
\begin{enumerate}
\item $ f|_{E} \colon (E, \omega') \rightarrow \C$ is a Lefschetz fibration with horizontal triviality and compact fibers. 
\item For every $z \in \C$, $\omega|_{\pi^{-1}(z)}=\omega'|_{\pi^{-1}(z)}$. 
\item $\omega'$ coincides with $\omega$ outside a collar neighborhood of the boundary $\del E$. 
\end{enumerate}
\end{proposition}

\begin{proof}
Set $L \coloneqq N_{B/\overline{F}_\infty}$ and take a Hermitian metric on it. 
Throughout the proof, we use the same notations as in Section \ref{section: Whitney sum} such as $L^{\oplus 2}(r)$ and $N_0(r_0,r_1)$. 
We also always equip a disc bundle $L^{\oplus 2} (r)$ for $r<1$ with the symplectic form $\Omega$ unless otherwise noted. 

Thanks to the conditions (\ref{condition: LF}) and (\ref{condition: LP symplectic}) in Definition \ref{def: LP}, it is sufficient to perturb $\omega$ near the boundary $\del E$ to complete the proof. 
The condition (\ref{condition: nbhd}) in the same definition enables us to take a tubular neighborhood $\nu'_{X}(B)$ of $B$ that is symplectomorphic to the interior of the disc bundle, $\Int (L^{\oplus 2}(r_1))$, of radius $r_1$, where $f$ agrees with projectivization $f_\nu \colon L^{\oplus 2}(r_1) \setminus B\rightarrow \CP^1$. 
Choose $r_0$ with $0<r_0<r_1$ and let $\nu_{X}(B)$ denote a tubular neighborhood of $B$ symplectomorphic to $\Int (L^{\oplus 2}(r_0))$. 
With these identifications, we will deform $\Omega$ on $N_0(r_0, r_1)$ instead of $\omega$ on $\nu'_{X}(B) \setminus \nu_{X}(B)$ in the rest of the proof. 

The idea of the deformation is similar to that of \cite[Lemma 9.3]{Keating_free}.
As we saw in Section \ref{section: Whitney sum}, 
$\Omega$ is exact on $N_0(r_0,r_1)$, i.e., $\Omega=d\lambda$; 
in particular, so is $\Omega$ on each fiber of $f_\nu|_{N_0(r_0,r_1)}$. 
Hence, one can take a fiberwise Liouville vector field $V'$ on $N_0(r_0,r_1)$.
% \red{such that 
%$$
%	i_{V'}\Omega|_{f_{\nu}^{-1}(z) \cap N_0(r_0,r_1)}=\lambda|_{f_{\nu}^{-1}(z) \cap N_0(r_0,r_1)}
%$$
%for every $z \in \CP^1$. }
Identify a collar neighborhood $\nu_{N_0(r_0,r_1)}(\del N_0(r_0))$ of $\del N_0(r_0)$ with $(-\tau, 0] \times (\del N_0(r_0) \cap f_\nu^{-1}([1:0])) \times \C$ by the diffeomorphism $\Phi \colon (-\tau, 0] \times (\del N_0(r_0) \cap f_\nu^{-1}([1:0])) \times \C \rightarrow \nu_{N_0(r_0,r_1)}(\del N_0(r_0))$ defined by 
$$
	\Phi(t, (z_0,0), z) = \phi_{t}^{V'}\left( \frac{z_0}{\sqrt{|z|^2+1}}, \frac{z_0z}{\sqrt{|z|^2+1}} \right), 
$$ 
where $\phi_{t}^{V'}$ with $t \in (-\tau, 0]$ denotes the fiberwise Liouville flow of $V'$. 
Let $\alpha_0 \coloneqq \lambda|_{\del N_0(r_0) \cap f_\nu^{-1}([1:0])}$ and 
$$
	\lambda' \coloneqq \left(e^{t}\alpha_{0}-\frac{1}{4\pi}d^{\C}\log(|z|^2+1) \right) -\Phi^{*}\lambda. 
$$
Notice that the $1$-form $\lambda'$ vanishes on every fiber of $f_\nu$. 
Now define the $2$-form $\eta'$ on $(-\tau, 0] \times (h^{-1}(\epsilon) \cap f_{\nu}^{-1}([1:0])) \times \C$ by $\eta' = d(\Phi^* \lambda)+d(\rho(t)\lambda')$, where $\rho \colon (-\tau, 0] \rightarrow [0,1]$ is a smooth cut-off function such that $\rho(t)=0$ near $t=-\tau$ and $\rho(t)=1$ near $t=0$.
%\red{It can be checked easily that the $2$-form $\eta'$ gives the same horizontal boundary as $\Omega$.  }
It follows from the choice of $\rho$ that the pushforward $\Omega'=\Phi_*(\eta')$ extends over the whole $N_0(r_0,r_1)$ in such a way that $\Omega'=\Omega$ away from $\nu_{N_0(r_0,r_1)}(\del N_0(r_0))$. 
Moreover, $\del N_{0}(r_0)$ is a connected component of the horizontal boundary of $N_{0}(r_0,r_1)$ for $f_\nu$ defined by $\Omega'$. 
%(There are two connected components of the horizontal boundary, but what we are interested in is only  $\del N_{0}(r_0)$.) 
Hence, the map $f_\nu \colon (N_0(r_0,r_1), \Omega')\rightarrow \C$ satisfies the condition for the horizontal triviality in Remark \ref{rem: horizontal}, which concludes that $\omega'$ is the desired perturbation of $\omega$ near the boundary $\del N_0(r_0)$. 
\end{proof}

\subsection{Relation between a fibered Dehn twist and a product ofDehn twists}\label{section: boundary-interior}

Let $(W, \omega)$ be a Liouville domain with Liouville vector field $X$ and set $\alpha \coloneqq (i_X\omega)|_{\del W}$. 
Fix an identification of a collar neighborhood $\nu_W(\del W)$ of $\del W$ with $((-\epsilon, 0] \times \del W, d({e^t}\alpha))$.
Suppose that all Reeb orbits of $\alpha$ are $2\pi$-periodic. 

A \textit{fibered Dehn twist} along $\del W$ is a symplectomorphism of $(W, \omega)$ defined by
$$
	\tau_{\del W} (p) = 
	\begin{cases}
	 (t, \phi_{g(t)}^{\alpha} (x)) & \mathrm{if\ } p=(t,x) \in  (-\epsilon, 0] \times \del W \cong \nu_W(\del W), \\
	p & \mathrm{otherwise},
	\end{cases}
$$
where $\phi_{t}^{\alpha}$ ($t \in \R$) denotes the Reeb flow of $\alpha$ and $g \colon (-\epsilon, 0] \rightarrow \R$ is a smooth function such that $g(t)$ equals $2\pi$ near $t=-\epsilon$ and $0$ near $t=0$. 
%\end{definition} 
A fibered Dehn twist has compact support by definition. 
Also, it is easy to see that $\tau_{\del W}$ is independent of the choice of $g$ up to Hamiltonian isotopy. 

Under a certain condition, a fibered Dehn twist can be factored into a product ofDehn twists: 

\begin{theorem}[{Gompf \cite[p.271]{Go} and Auroux \cite[p.6]{Au}}]\label{thm: boundary-interior}
Let $(X, \omega)$ be a closed integral symplectic manifold admitting a Lefschetz pencil $f \colon X \setminus B \rightarrow \CP^{1}$ and $W$ the complement of a tubular neighborhood of the base locus $B$ in the closure of a regular fiber of $f$. 
Then, a fibered Dehn twist $\tau_{\del W}$ along the boundary $\del W$ is symplectically isotopic to a product of Dehn twists along Lagrangian spheres in $(W, \omega|_{W})$. 
In other words, $[\tau_{\del W}] \in \pi_0(\Symp_{c}(W,\omega|_{W}))$ factors into a product of isotopy classes of Dehn twists. 
\end{theorem}

In Appendix \ref{appendix: boundary-interior}, we give a proof of Theorem \ref{thm: boundary-interior} from the viewpoint of Lefschetz--Bott fibrations, restricting the theorem to the case where $X$ is complex projective.

\section{Symplectic isotopy problem on del Pezzo surfaces}\label{section: isotopy}

To prove the main theorem (Theorem \ref{thm: relation}), we will symplectically identify the complements of homologous symplectic submanifolds of (real) dimension $2$ in a K\"{a}hler surface, that is, a K\"{a}hler manifold of complex dimension $2$. 
In general, such complements are not symplectomorphic or even diffeomorphic each other.  
In some special cases, a solution to the \textit{symplectic isotopy problem} provides a symplectomorphism between the complements.  

\begin{problem}[Symplectic isotopy problem]\label{prob: isotopy}
Let $(M,\omega, J)$ be a K\"ahler surface and $S$ a symplectic submanifold of $(M, \omega)$ of $\dim_{\R}S=2$. 
Then, is $S$ symplectically isotopic to a complex curve in $(M,J)$?
\end{problem}

Suppose that the problem is solved affirmatively for $(M, \omega, J)$ and also suppose that any two homologous smooth complex curves in $(M, \omega, J)$ are symplectically isotopic. 
Then, this isotopy yields the desired symplectomorphism; see Proposition \ref{prop: del Pezzo symplectomorphism} below.  
The goal of this section is to solve this problem for del Pezzo surfaces via holomorphic curve techniques.

\subsection{Preliminaries of holomorphic curves}\label{section: holomorphic curves}

\subsubsection{Moduli spaces of holomorphic curves}

Let $(\Sigma, j)$ be a smooth Riemann surface and $(M,J)$ an almost complex manifold. 
A map $u \colon (\Sigma,j) \rightarrow (M,J)$ is called a \textit{pseudo-holomorphic curve}, or \textit{$J$-holomorphic curve} if $u$ is $(j,J)$-holomorphic, that is, $du \circ j=J \circ du$. 
$J$-holomorphicity can also be defined in the case where $(\Sigma, j)$ is a nodal Riemann surface: 
For the normalization $\varphi \colon \tilde{\Sigma}=\coprod_{i} \tilde{\Sigma}_{i} \rightarrow \Sigma$ of $\Sigma$, 
set $\varphi_{i}=\varphi|_{\Sigma_{i}}$. 
A map $u \colon (\Sigma, j) \rightarrow (M,J)$ is said to be \textit{$J$-holomorphic} if each $u \circ \varphi_{i}$ is $J$-holomorphic. 
Let us write $(\Sigma, j, \bm{x})$ for a marked nodal Riemann surface with $m$ marked points $\bm{x}=\{x_1, \ldots, x_m\} \subset \Sigma$. 
A \textit{stable} $J$-holomorphic curve $u \colon (\Sigma, j, \bm{x}) \rightarrow (M,J)$ is a $J$-holomorphic curve with the finite automorphism group $\mathrm{Aut}(u)=\{\phi \in \mathrm{Aut}(\Sigma,j, \bm{x}) \mid u \circ \phi =u \}$. 
Fix $A \in H_{2}(M; \Z)$ and a finite subset $\bm{z}=\{z_1, \ldots, z_m\}$ of $M$. 
We denote the moduli space of (unparametrized) $J$-holomorphic stable curves of genus $g$ in the class $A$ passing through $\bm{z}$ by 
$$\overline{\M}_{g ,m}(A;J;\bm{z})$$ 
and, given a symplectic structure $\omega$ on $M$, set 
$$
	\overline{\M}_{g,m}(A; \mathcal{J}_{\tau}; \bm{z})=\bigcup_{J \in \mathcal{J}_{\tau}(M,\omega)} \overline{\M}_{g,m}(A;J;\bm{z}) \times \{J\},  
$$
where $\mathcal{J}_{\tau}(M,\omega)$ is the set of smooth $\omega$-tame almost complex structures on $M$. 
Here an element of $\overline{\M}_{g,m}(A;J;\bm{z})$ is an isomorphism class of a stable $J$-holomorphic curve $u \colon (\Sigma, j, \bm{x}) \rightarrow (M,J)$ such that $(\Sigma, j)$ is a nodal Riemann surface of genus $g$, $[u(\Sigma)]=A$ and $u(x_{i}) = z_i$ for every $i$. 
When $m=0$, we will suppress $m$ and $\bm{z}$ from the notations of the moduli spaces. 
Also put 
$$
	\M_{g ,m}(A;J;\bm{z})=\{[\Sigma, j, u, \bm{x}] \in \overline{\M}_{g,m}(A; \mathcal{J}_{\tau}; \bm{z}) \mid \Sigma \textrm{ is smooth}\}, 
$$
$$
	\M_{g,m}(A; \mathcal{J}_{\tau}; \bm{z})=\bigcup_{J \in \mathcal{J}_{\tau}(M,\omega)} \M_{g,m}(A;J;\bm{z}) \times \{J\}.
$$
We often call an element of $\M_{g ,m}(A;J;\bm{z})$ and $\M_{g ,m}(A; \mathcal{J}_{\tau};\bm{z})$ a \textit{smooth} holomorphic curve.

As we will discuss a local structure of moduli spaces, we need to topologize these moduli spaces. 
We endow $\overline{\M}_{g,m}(A; \mathcal{J}_{\tau}; \bm{z})$ with the \textit{$C^{0}$-topology} and $\overline{\M}_{g,m}(A;J;\bm{z})$ with the subspace topology (see \cite[Section 3]{Sie}). 
Although we omit the definition of the $C^0$-topology, one of the things to keep in mind is that with this topology the moduli space $\overline{\M}_{g,m}(A;J;\bm{z})$ is compact and Hausdorff. 

\subsubsection{Holomorphic curves in dimension four} 

Holomorphic curves dealt with in this paper are only over symplectic $4$-manifolds. 
Such holomorphic curves are well studied, and here we will collect results for them. 

Two complex curves in a complex surface intersect positively. 
The study of local structures of holomorphic curves shows an almost complex analogue of this phenomenon: 

\begin{theorem}[{Positivity of intersections \cite{McD}, \cite{MW}}]\label{thm: positivity}
Let $(M,J)$ be an almost complex manifold of $\dim_{\R}M=4$. 
Suppose that $u_0 \colon (\Sigma_0, j_0) \rightarrow (M,J)$ and $u_1 \colon (\Sigma_1, j_1) \rightarrow (M,J)$ are smooth connected $J$-holomorphic curves whose images are not identical. 
Then, $u_0$ and $u_1$ intersect at most finitely many points, and the homological intersection number $[u_0(\Sigma_0)] \cdot [u_1(\Sigma_1)]$ satisfies 
$$
	[u_0(\Sigma_0)] \cdot [u_1(\Sigma_1)] \geq |\{ (z_0,z_1) \in \Sigma_0 \times \Sigma_1 \mid u_0(z_0)=u_1(z_1)\}|, 
$$
with equality if and only if all the intersections are transverse.
\end{theorem}

Suppose that $u \colon (\Sigma, j) \rightarrow (M,J)$ is a \textit{simple} $J$-holomorphic curve, i.e., $u$ cannot split into a non-trivial holomorphic branched covering $(\Sigma,j) \rightarrow (\Sigma', j')$ and a $J$-holomorphic curve $u' \colon (\Sigma', j') \rightarrow (M,J)$. 
Set 
$$
	\mathcal{D}(u)=\{ \{z,z'\} \in \mathcal{P}(\Sigma) \mid u(z)=u(z'),  z \neq z'\} \textrm{\quad and \quad }
\mathcal{C}(u)=\{z \in \Sigma \mid du(z)=0\},$$
where $\mathcal{P}(\Sigma)$ is the power set of $\Sigma$. 
Let $\delta(u;z,z')$ denote the \textit{local interesection index} for $\{z,z'\} \in \mathcal{D}(u)$ of $u$ (see \cite[Section 2.10]{Wen_Lecture}).
Also let $\delta(u;z)$ denote the \textit{virtual number of double points} of $u$ at $z \in \mathcal{C}(u)$ (see \cite[Theorem 7.3]{MW} and also \cite[Section 10]{Milnor}). 
We define the integer $\delta(u)$ by 
\begin{equation}\label{eqn: delta}
	\delta(u) \coloneqq \sum_{\{z,z'\} \in \mathcal{D}(u)}\delta(u;z,z')+\sum_{z \in \mathcal{C}(u)}\delta(u;z) \in \Z. 
\end{equation}
By definition, $u \colon (\Sigma, j) \rightarrow (M,J)$ is an embedding with the smooth domain $(\Sigma,j)$ if and only if $\delta(u)=0$. 
We also note that $\delta(u) \geq 0$ for any $J$-holomorphic curve $u$. 

\begin{theorem}[{Adjunction formula \cite{McD}, %\cite[Theorem 7.3]
\cite{MW}}]\label{thm: adjunction}
Let $(M,J)$ be an almost complex manifold of $\dim_{\R}M=4$ and $u \colon (\Sigma,j) \rightarrow (M,J)$ a simple $J$-holomorphic curve. 
Also let $\tilde{\Sigma}=\coprod_{j=1}^{d} \tilde{\Sigma}_{j} \rightarrow \Sigma$ be the normalization of $\Sigma$. 
Then we have 
\begin{equation}\label{genus formula}
	2\delta(u) = [u(\Sigma)]^2-c_1(TM)([u(\Sigma)])+\sum_{j=1}^{d}\chi(\tilde{\Sigma}_{j}). 
\end{equation}
\end{theorem}

The moduli space $\M_{g,m}(A;J;\bm{z})$ is not a manifold in general. 
One sufficient condition for $\M_{g,m}(A;J;\bm{z})$ to be a manifold is that every element $[\Sigma, j, u, \bm{x}] \in \M_{g,m}(A;J;\bm{z})$ is \textit{Fredholm regular} (see e.g. \cite{Wen_Lecture} for its definition). 
In dimension $4$, this regularity is guaranteed by a homological condition: 

\begin{theorem}[Automatic regularity \cite{HLS}]\label{thm: HLS}
Let $(M,J)$ be an almost complex manifold of $\dim_{\R}M=4$ and $[\Sigma, j, u, \bm{x}] \in \mathcal{M}_{g,m}(A;J;\bm{z})$ an immersed $J$-holomorphic curve. 
If $$c_1(TM)(A)>m,$$ then $[\Sigma, j, u, \bm{x}] $ is Fredholm regular. 
%\red{In particular, $\mathcal{M}_{g,m}(A;J;\bm{z})$ is a manifold around $[\Sigma, j, u, \bm{x}]$.}
\end{theorem}

%The above theorem implies that for $[\Sigma, j, u, \bm{x}] \in \M_{g,m}(A;J;\bm{z})$ satisfying the assumption of the theorem, the projection $ \M_{g,m}(A; \bm{z}) \rightarrow \mathcal{J}(M)$ is a local submersion near $[\Sigma, j, u, \bm{x},J]$. 

%\begin{corollary}
%The space of almost complex structures that are regular for a homology class $A$ with $c_1(TM)(A)>0$ is path-connected. 
%We denote this space by $\mathcal{J}_{\mathrm{reg}}(M;A)$.
%\end{corollary}

Let $(\Sigma, j)$ be a nodal Riemann surface. 
A $J$-holomorphic curve $u \colon (\Sigma,j) \rightarrow (M,J)$ is said to be \textit{nodal} if $u$ is an embedding, that is, for the regularization $\varphi \colon \tilde{\Sigma}=\coprod_{j} \tilde{\Sigma}_{j} \rightarrow \Sigma$, $u \circ \varphi$ is an embedding on each $\tilde{\Sigma}_{j}$ and the image $(u \circ \varphi) (\tilde{\Sigma})$ has distinct tangent spaces at all the nodes. 
The following theorem of Sikorav tells us that the stratum consisting of nodal holomorphic curves is at least (real) codimension $2$ in the compactified moduli space: 

\begin{theorem}[{\cite[Corollary 1.4]{Sik}}]\label{thm: Sikorav}
Let $M$ be a $4$-manifold with an almost complex structure $J$ for which all elements of $\M_{g,m}(A;J; \bm{z})$ are Fredholm regular.
Let $[\Sigma, j, u, \bm{x}]$ be an element of $\overline{\M}_{g,m}(A;J;\bm{z})$ such that $u$ is a nodal curve with $k$ nodes and $\bm{x}$ does not contain nodes. 
We denote the normalization of $\Sigma$ by $\varphi \colon \tilde{\Sigma}=\coprod_{i=1}^{r} \tilde{\Sigma}_{i} \rightarrow \Sigma$ and set $u_i = u \circ \varphi|_{\tilde{\Sigma}_{i}} \colon \tilde{\Sigma}_{i} \rightarrow M$. 
Suppose that it holds that 
$$
	c_1(TM)([u_i(\tilde{\Sigma}_i)])>|\bm{z} \cap u_i(\tilde{\Sigma}_i)|
$$
for all $i=1,\ldots,r$.
Then, a neighborhood of $[\Sigma, j, u, \bm{x}]$ in $\overline{\M}_{g,m}(A;J;\bm{z})$ is homeomorphic to an open neighborhood of the origin in $\C^{\dim_{\C}\M_{g,m}(A;J;\bm{z})}$ with the expected complex dimension of the moduli space $\M_{g,m}(A;J;\bm{z})$ 
$$
\dim_{\C} \M_{g,m}(A;J;\bm{z}) =c_1(TM)(A)+g-1-m.
$$ 
The subset parametrizing nodal $J$-holomorphic curves is a union of complex coordinate hyperplanes $\{(\xi_1, \ldots, \xi_{k}) \in \C^k \mid \xi_i=0\} \times \C^{\dim_{\C} \M_{g,m}(A;J;\bm{z})-k}$. 
\end{theorem}

In addition to Sikorav's original paper \cite{Sik}, the reader is referred to \cite[pp.998--1000]{ST} for a sketch of the proof of Theorem \ref{thm: Sikorav}.

\subsection{Symplectic isotopy problem}\label{section: solution}
%\note{our proof is inspired by \cite[Theorem 2]{Bar} and \cite[Thereom]{Sik}}

We shall now discuss Problem \ref{prob: isotopy} for del Pezzo surfaces. 
Let us begin by recalling the definition of del Pezzo surface. 

\begin{definition}
A \textit{del Pezzo surface} $M$ is a smooth complex projective surface with ample anti-canonical class $-K_{M}$. 
The number $K_{M}^{2} \coloneqq K_{M} \cdot K_{M}$ is called the \textit{degree} of $M$.
\end{definition}

It is well known that a del Pezzo surface of degree $d$ is diffeomorphic to either 
$\CP^2 \# \overline{\CP}^2$ or $S^2 \times S^2$ if $d=8$; otherwise, 
$\CP^2 \# (9-d)\overline{\CP}^2$ (see \cite[Theorem 24.4]{Manin}). 
Moreover, according to \cite[Remark 24.4.1]{Manin}, all del Pezzo surfaces of degree $d$ for $5 \leq d \leq 7$ are biholomorphic.

The following theorem is the main result in this section. 

\begin{theorem}\label{thm: isotopy}
Let $M$ be a del Pezzo surface diffeomorphic to $\CP^2 \# n \overline{\CP}^2$ ($0 \leq n \leq 8$) and $\omega$ a K\"ahler form on $M$. 
Suppose that $S$ is a symplectic submanifold of $(M,\omega)$, of $\dim_{\R}S=2$, homologous to an anti-canonical divisor. 
Then, $S$ is symplectically isotopic to a smooth complex curve. 
Furthermore, two such symplectic submanifolds on $M$ are mutually symplectically isotopic if $0 \leq n \leq 6$. 
\end{theorem}

%We remark that a K\"ahler form $\omega$ in the theorem is not necessarily monotone. 

This theorem has been partially proven by Sikorav \cite[Theorem 1.5]{Sik} and Shevchishin \cite[Theorem 1]{Shev} for the case $n=0$ and Siebert and Tian \cite[Theorem B]{ST} for the case $n=0,1$; see also \cite[Proposition 3.6]{LiMak} where Li and Mak gave a generalization of the theorem to symplectic divisors in the latter case. 

We will prove Theorem \ref{thm: isotopy} by combining several lemmas.  
Before getting into the details, let us fix the notations. 
Throughout the rest of this section, we set 
$$
	M(n) \coloneqq \CP^2 \# n \overline{\CP}^2
$$
and write $J(n)$ for an integrable almost complex structure on $M(n)$ and $\omega_0$ for a K\"{a}hler form on $M(n)$ with respect to $J(n)$. 
We denote the homology classes of a complex line and $n$ disjoint exceptional spheres in $M(n)$ by 
$$
H, E_{1}, \ldots, E_{n} \in H_{2}(M(n); \Z),
$$
respectively. 
The homology class $K_n \in H_2(M(n); \Z)$ is defined to be 
$$
	K_n \coloneqq -3H+\sum_{i=1}^{n}E_i. 
$$
Given $A\in H_2(M(n); \Z)$, we define the set $\mathcal{J}_{\textrm{reg}}(M(n); A)$ of $\omega_0$-tame almost complex structures $J$ on $M(n)$ for which every simple $J$-holomorphic curve in the class $A$ is Fredholm regular. 
According to \cite[Section 6.2]{MSbook} and \cite[Chapter 5]{Chang}, given $8-n$ points $\bm{z}$ on $M(n)$, there exists a residual subset $\mathcal{J}_{\mathrm{reg}}^{*}(n; \bm{z})$ of $\mathcal{J}_{\tau}(M(n), \omega_0)$ satisfying the following properties: 
\begin{itemize}
\item  $\mathcal{J}_{\mathrm{reg}}^{*}(n; \bm{z})$ is path-connected.

\item For every  $J \in \mathcal{J}_{\mathrm{reg}}^{*}(n; \bm{z})$ and any decomposition $-K_n=A_1+\cdots + A_m$ consisting of non-trivial spherical classes $A_i$, every simple $J$-holomorphic sphere representing $A_i$ is Fredholm regular.

\item For any $J \in \mathcal{J}_{\mathrm{reg}}^{*}(n; \bm{z})$, all $J$-holomorphic spheres in the class $-K_n$ passing through $\bm{z}$ have only nodes as singularities. 
\end{itemize}

\begin{lemma}\label{lem: positivity}
Let $J$ be an element of $\mathcal{J}_{\mathrm{reg}}(M(n); dH-\sum_{i=1}^{n}e_i E_i)$ and  $u \colon (\Sigma,j) \rightarrow (M(n),J)$ a non-constant smooth simple $J$-holomorphic curve of genus at most $1$. 
Suppose that there is no $i_0 \in \{1,\ldots, n\}$ such that $dH-\sum_{i=1}^{n}e_iE_i = E_{i_0}$.
If $[u(\Sigma)]=dH-\sum_{i=1}^{n}e_i E_i \in H_2(M(n); \Z)$, then, $d > 0$ and $e_{i} \geq 0$. 
\end{lemma}

\begin{proof}
We first note that $J$-holomorphic spheres in the classes $H$ and $E_i$ always exist (see \cite[Theorem 5.1]{WenBook} and its proof), and by assumption $u$  does not coincide with any $J$-holomorphic sphere in each class $E_i$. 
Also, it is possible to take a $J$-holomorphic sphere in the class $H$ that is not identical to $u$ because of the positivity of the index of such a holomorphic sphere.
Hence, the positivity of intersections (Theorem \ref{thm: positivity}) shows that $d \geq 0$ and $e_i \geq 0$ for each $i$. 
Suppose $d=0$. 
Then, the (real) dimension of the moduli space of simple $J$-holomorphic curves of genus $g(\Sigma)$ in the class $-\sum_{i=1}^{n}e_i E_i$ is given by 
$$
	-\chi(\Sigma)+2\left( c_{1}(TM(n))\left(-\sum_{i=1}^{n}e_{i}E_{i}\right) \right)=2\left(g(\Sigma)-1-\sum_{i=1}^{n}e_i \right).
$$
With $g(\Sigma)=0$ or $1$, we have $2\left(g(\Sigma)-1-\sum_{i=1}^{n}e_i \right) \leq -2\sum_{i=1}^{n}e_i$.
Note that $u$ is non-constant and there must be at least one positive $e_{i}$. 
Hence, the dimension of the moduli space is negative, which contradicts the existence of $u$. 
Thus, $d>0$.
\end{proof} 

The next corollary follows from the above lemma and its proof immediately. 

\begin{corollary}\label{cor: positivity}
Let $u \colon (\Sigma, j) \rightarrow (M(n), J)$ be a non-constant smooth simple $J$-holomorphic curve of genus at most $1$ with $J \in \mathcal{J}_{\mathrm{reg}}(M(n); dH-\sum_{i=1}^{n}e_i E_i)$. 
If $u(\Sigma)$ does not intersect a $J$-holomorphic sphere in the class $H$, then $[u(\Sigma)]=E_{i_0}$ for some $i_0$.
%In fact, such a curve $u(\Sigma)$ is an embedded sphere according to the formula (\ref{genus formula}).
\end{corollary}

\begin{lemma}\label{lem: genus one}
For $J\in \mathcal{J}_{\mathrm{reg}}(M(n); dH-\sum_{i=1}^{n}e_i E_i)$, let $u \colon (\Sigma,j) \rightarrow (M(n),J)$ be a non-constant smooth simple $J$-holomorphic curve with $[u(\Sigma)]=dH-\sum_{i=1}^{n}e_i E_i$. 
If $0 \leq d \leq 3$ and $g(\Sigma)=1$, then we have $d =3$ and $e_i \in \{ 0,1\}$. 
\end{lemma}

\begin{proof}
By the formula (\ref{genus formula}) in Theorem \ref{thm: adjunction}, we have 
$$
	2\delta(u)=d(d-3)+\sum_{i=1}^{n}e_i(1-e_i). 
$$
As $\delta \geq 0$, $d(d-3)\leq0$ and $\sum_{i=1}^{n}e_i(1-e_i) \leq 0$, we have $d \in \{0,3\}$ and $e_i \in \{0,1\}$. 
If $d=0$, then Corollary \ref{cor: positivity} implies that $[u(\Sigma)]=E_{i_0}$ for some $i_0$ and $g(\Sigma)=0$, which is a contradiction. 
\end{proof}

\begin{lemma}\label{lem: simpleness}
Let $(J_{\nu})$ be a sequence of almost complex structures in $\mathcal{J}_{\textrm{reg}}^{*}(n;\bm{z})$ and $(u_{\nu})$ a sequence of embedded $J_{\nu}$-holomorphic curves of genus $1$ in the class $-K_n$ passing through generic $8-n$ points $\bm{z}$ in $M(n)$. 
Suppose that $J_{\nu} \rightarrow J_{\infty} \in \mathcal{J}_{\textrm{reg}}^{*}(n; \bm{z})$ in the $C^{0}$-topology and $u_{\nu} \rightarrow u_{\infty}$ in the $C^{0}$-topology. 
The limit stable curve $u_{\infty}$ is written as $(u_{\infty, 1}, \ldots, u_{\infty, N})$, where each $u_{\infty, a}$ is a $J_{\infty}$-holomorphic curve defined on a respective irreducible component $\tilde{\Sigma}_{\infty, a}$ of the limit nodal curve $\Sigma_{\infty}$.
Then, all non-constant curves $u_{\infty,a}$ are simple. 
\end{lemma}

\begin{proof}
%The proof will be done by a homological argument. 
For notational convenience, set $C_{\infty} \coloneqq u_{\infty}(\Sigma_{\infty})$ and $C_{\infty,a} \coloneqq u_{\infty,a}(\tilde{\Sigma}_{\infty,a})$.
Suppose that $[C_{\infty, a}]=d_{a}H-\sum_{i=1}^{n}e_{a,i}E_{i}$. 
Let $m_a \in \Z_{\geq 0}$ be the multiplicity of $C_{\infty,a}$ and 
$k_a\in \Z_{\geq 0}$ the number of points of $C_{\infty,a} \cap \bm{z}$. 
In view of the dimension of the moduli space, for each $a$, we have $2k_a \leq d_{a}(d_{a}+3)-\sum_{i=1}^{n}e_{a,i}(e_{a,i}+1)$. 
We will show that if there exists $m_a \geq 2$, then we have
$$
	2\sum_{a=1}^{N} k_{a} \leq \sum_{a=1}^{N}d_{a}(d_{a}+3)-\sum_{i=1}^{n}\left( \sum_{a=1}^{N}e_{a,i}(e_{a,i}+1) \right) <2(8-n), 
$$
which contradicts the fact that $C_{\infty}$ passes through the fixed $8-n$ points $\bm{z}$.

Observe that 
$
	\sum_{a=1}^{N}e_{a,i_{0}}(e_{a,i_{0}}+1) \geq 2 
$
for any fixed $i_{0}$. 
Indeed,  
if all $e_{a,i_{0}}$ are $0$ or $-1$, then $\sum_{a=1}^{N}m_a e_{a,i_{0}} \leq 0<1$, contrary to the fact that $\sum_{a=1}^{N}m_a e_{a,i_{0}}=1$. 
Hence, there exists $e_{a_0,i_{0}}$ that is not $0$ or $-1$, and 
%As $p(p+1) \geq 0$ for any $p \in \Z$, 
we have 
\begin{align}\label{inq: e}
\sum_{a=1}^{N}e_{a,i_{0}}(e_{a,i_{0}}+1) \geq 1 \cdot (1+1)=2.
\end{align} 

In what follows, without loss of generality, we may assume that $m_1 \geq 2$. 

\noindent
\textbf{Case 1.} $d_1=0$. 

%In this case, we have $[C_{\infty,1}]=-\sum_{i=1}^{n} e_{1,i}E_{i}$.
%If $e_{1, i}=0$ for all $i$, the homology class $[C_{\infty,1}]$ is null-homologous. The energy of $C_{\infty,1}$ is $0$, and $C_{\infty,1}$ is constant, which is a contradiction. 
%Since $C_{\infty,1}$ is a non-constant holomorphic curve, at least one of $e_{1,i}$, say $e_{1, i_{0}}$, is not equal to $0$. 
%If $e_{1, i_{0}}>0$, then there exists $a$ with $e_{a,i_{0}} \leq -1$ since  $\sum_{a}m_{a}e_{a, i_{0}}=1$. 
%Moreover, as the dimension of the moduli space of $J_{\infty}$-holomorphic sphere in $[C_{\infty,1}]$ is $-\sum_{i=1}^{n}e_{1,i}(e_{1,i}+1)$, we may assume that $e_{1,i_0}=-1$. 
%In view of $\sum_{a}m_{a}e_{a, i_{0}}=1$, there exists $e_{1,i_{0}} \leq -1$.
In this case, by Corollary \ref{cor: positivity}, we have $[C_{\infty,1}]=E_{i_{0}}$ for some $i_{0}$. 
Hence we obtain 
\begin{align}\label{inq: e2}
	\sum_{2 \leq a \leq N} m_{a}e_{a,i_{0}}=1-m_{1}e_{1,i_{0}} \geq 3
\end{align}
for the fixed $i_0$, which shows that for some $a_{0} \geq 2$,
$$
	e_{a_{0}, i_{0}}>0.
$$
Thus, it follows from Corollary \ref{cor: positivity} again that $d_{a_{0}}$ %\note{Is our $J$ generic for this Fredholm regular problem?} 
must be positive. 
This leads to the following two cases. %according to the numbers $d_{a_{0}}$ and $m_{a_{0}}$. 

\noindent
\textbf{Case 1-(i).} $d_{a_{0}} =1,2$. 

In this case, observe that $\sum_{a=1}^{N}d_a(d_a+3) \leq 14$. 
Thus, 
\begin{align*}
	\sum_{a=1}^{N}d_a(d_a+3)-\sum_{i=1}^{n}\sum_{a=1}^{N}e_{a,i}(e_{a,i}+1) \leq  \ 14-2n<16-2n.
\end{align*}

\noindent
\textbf{Case 1-(ii).} $d_{a_0}=3$. 

As $d_{a_0}=3$, we have $m_{a_0}=1$ and hence $d_{a}=0$ for all $a \neq a_0$. 
Combining Corollary \ref{cor: positivity} with the inequality (\ref{inq: e2}) shows that  $e_{a, i_{0}} \in \{-1,0\}$ for all $a \neq a_0$ and $e_{a_0, i_0} \geq 3$.
Then, with the inequality (\ref{inq: e}) we have
\begin{align*}
	&\ \sum_{a=1}^{N}d_a(d_a+3)-\sum_{a=1}^{N}\sum_{i=1}^{n}e_{a,i}(e_{a,i}+1) \\
	 \leq &\
	\sum_{a=1}^{N}d_a(d_a+3)-\sum_{i=1, i \neq i_0}^{n}\sum_{a=1}^{N} e_{a,i}(e_{a,i}+1)-\sum_{a=1}^{N}e_{a,i_0}(e_{a, i_0}+1) \\
	 \leq & \ 3 \cdot (3+3) - (n-1) \cdot 2 - 3(3+1) = 8-2n<16-2n. 
\end{align*}
This completes Case 1.

\noindent
\textbf{Case 2.} $d_1=1$. 

Notice that $2 \leq m_{1} \leq 3$ in this case. 
Suppose that $m_1=2$. 
As $d_{a} \geq 0$ for any $a$ by Lemma \ref{lem: positivity}, there exists a unique $a_0>1$ satisfying $d_{a_{0}}=1$, and the rest of $d_a$ must be $0$. 
Hence, we obtain $\sum_{a}d_{a}(d_{a}+3) \leq 8$, which also holds for the case $m_1=3$. 
Therefore, we conclude that 
$$
	\sum_{a=1}^{N}d_{a}(d_{a}+3)-\sum_{i=1}^{n}\left( \sum_{a=1}^{N}e_{a,i}(e_{a,i}+1)\right) \leq 8- 2n < 16-2n, 
$$
which finishes the proof of the lemma.
\end{proof}

\begin{lemma}\label{lem: N=1}
Let $u_{\infty}=(u_{\infty,1}, \ldots, u_{\infty,N}) \colon (\Sigma_{\infty}, j_{\infty}) \rightarrow (M_{n}, J_{\infty})$ be the limit curve in Lemma \ref{lem: simpleness}. Then, each $u_{\infty,a}$ is a non-constant curve that is either an embedding of a genus $1$ curve or a nodal sphere with one node. 
In particular, $N=1$. 
\end{lemma}

\begin{proof}
We shall use the same notations as in the proof of Lemma \ref{lem: simpleness} and call each $u_{\infty, a}$ a \textit{component} of $u_{\infty}$. 
The following argument is inspired by the proof of \cite[Theorem 2]{Bar}. 
%Each holomorphic curve $u_{\infty, i}$ will be referred to as a component of $u_{\infty}$. 

First assume that all components $u_{\infty,a}$ are non-constant. 
Let $k_{a}$ be the number of points of $\bm{z}$ where $C_{\infty,a}$ passes, and put $\bm{z}_{a}=\bm{z} \cap C_{\infty,a}$. 
Let $\mathcal{M}^{*}_{g_a,k_a}(d_{a}H-\sum_{i=1}^{n}e_{a,i}E_{i}; J_{\infty}; \bm{z}_{a})$ be the subspace of $\mathcal{M}_{g_a,k_a}(d_{a}H-\sum_{i=1}^{n}e_{a,i}E_{i}; J_{\infty}; \bm{z}_{a})$ consisting of simple $J_{\infty}$-holomorphic curves. 
The expected (real) dimensions of the moduli spaces $\mathcal{M}^{*}_{0,k_a}(d_{a}H-\sum_{i=1}^{n}e_{a,i}E_{i}; J_{\infty}; \bm{z}_{a})$ and $\mathcal{M}^{*}_{1,k_a}(d_{a}H-\sum_{i=1}^{n}e_{a,i}E_{i}; J_{\infty}; \bm{z}_{a})$ are given by 
%\begin{eqnarray*}
%	& & 2\left( \frac{\chi(\Sigma)}{2}+c_{1}(M)\left(d_{j}H-\sum_{i=1}^{m}e_{i,j}E_{i}\right) \right)-
%	2\left(3d'-\sum_{i=1}^{m}e_{i}'\right) \\ 
%	& = & \begin{cases}
%			2+2\left(3(d_{j}-d')-\sum_{i=1}^{m}(e_{i}^{j}-e_{i}')\right) & (g(\Sigma)=0), \\ 
%			2\left(3(d_{j}-d')-\sum_{i=1}^{m}(e_{i}^{j}-e_{i}')\right) & (g(\Sigma)=1),
%			\end{cases}
%\end{eqnarray*}
\begin{eqnarray*}
	& & 2\left( \frac{\chi(\tilde{\Sigma}_{\infty, a})}{2}\dim_{\C}M+c_{1}(M)\left(d_{a}H-\sum_{i=1}^{n}e_{a,i}E_{i}\right) \right)-3\chi(\tilde{\Sigma}_{\infty, a})-2k_a \\ 
	& = & \begin{cases}
			2\left(-1+3d_{a}-\sum_{i=1}^{n}e_{a,i}-k_a\right) & \textrm{if}\ \ g(\tilde{\Sigma}_{\infty, a})=0, \\ 
			2\left(3d_{a}-\sum_{i=1}^{n}e_{a,i}-k_a\right) & \textrm{if} \ \ g(\tilde{\Sigma}_{\infty, a})=1.
			\end{cases}
\end{eqnarray*}
Since $J_{\infty}$ is generic, this shows that the moduli space $\mathcal{M}^{*}_{0,k_a}(d_{a}H-\sum_{i=1}^{n}e_{a,i}E_{i}; J_{\infty}; \bm{z}_{a})$  (resp. $\mathcal{M}^{*}_{1,k_a}(d_{a}H-\sum_{i=1}^{n}e_{a,i}E_{i}; J_{\infty}; \bm{z}_{a})$) is non-empty only if 
$k_a \leq 3d_{a}-\sum_{i=1}^{n}e_{a,i}-1$ (resp. $k_a\leq 3d_{a}-\sum_{i=1}^{n}e_{a,i}$). 
For each point $z$ of $\bm{z}$, more than one curve $C_{\infty,a}$ may pass through $z$. 
Hence if all $C_{\infty, a}$ are spheres, we have 
$$
	8-n \leq \sum_{a=1}^{N}k_{a} \leq \sum_{a=1}^{N}\left(3d_{a}-\sum_{a=1}^{n}e_{a,i}-1\right)=9-n-N, 
$$
which implies that $N=1$. 
Suppose that one of $C_{\infty,a}$, say $C_{\infty,1}$, has genus $1$. 
In this case, 
$$
	8-n \leq \sum_{a=1}^{N}k_{a} \leq \left( 3d_{1}-\sum_{i=1}^{n}e_{1,i} \right)+\sum_{a=2}^{N}\left(3d_{a}-\sum_{i=1}^{n}e_{a,i}-1\right)=9-n-(N-1), 
$$
which shows $N \leq 2$. 
Note that by Lemma \ref{lem: genus one}, $d_1=3$ and $e_{1,i} \in \{0,1\}$ for all $i$. 
If $N=2$, $d_2$ must be $0$, and Corollary \ref{cor: positivity} implies $[C_{\infty,2}]=E_{i_0}$ for some $i_0$. 
Hence, we have $e_{1,i_0} =2$, which is a contradiction. Thus, $N=1$. 
As the domain of $u_{\infty}=u_{\infty,1}$ is either a smooth curve of genus $1$ or a nodal sphere with one node, the lemma follows from the formula (\ref{genus formula}) and the definition of $\mathcal{J}_{\mathrm{reg}}^{*}(n;\bm{z})$. %and \cite{Chang}. 

In general, some $u_{\infty, a}$ might be a constant curve, which does not happen actually: 
Suppose $u_{\infty}$ has a contant component. 
If a unique non-constant component is a genus $1$ curve, then there is a constant sphere $u_{\infty, a}$ having two marked points in $\bm{x}$ by the stability condition. 
However, these points are mapped to the same point of $\bm{z}$, a contradiction. 
Next assume that a unique non-constant curve, say $u_{\infty,1}$, is a nodal sphere with one node. 
In order to prevent $u_{\infty}$ from having a constant sphere with two marked points, the domain of $u_{\infty}$ needs to be a nodal Riemann surface $\Sigma_{\infty}$ that consists of two spheres $\Sigma_{\infty,1}$ and $\Sigma_{\infty, 2}$ glued together at two nodes; $\Sigma_{\infty,1}$ has $7-n$ marked points and $\Sigma_{\infty, 2}$ has one marked point; see Figure \ref{fig: nodal sphere}.
When $n=8$, there cannot be such a nodal Riemann surface with marked points in the first place; otherwise, the expected dimension of the moduli space of such $J_{\infty}$-holomorphic curves $u \colon \Sigma_{\infty} \rightarrow M$ is negative. 
This is contrary to the existence of $u_{\infty}$.
Therefore, we conclude that $u_{\infty}$ has no constant component.  
%This can be seen easily in view of the stability condition and an argument on the expected dimensions of the moduli spaces of holomorphic spheres. 
%Indeed, if a unique non-constant curve, say $u_{\infty,1}$, is a genus $1$ curve, then there must be a constant sphere with $2$ marked points going into a point of $\bf{z}$. 
\end{proof}

\begin{figure}[ht]
	\centering
	\begin{overpic}[width=80pt,clip]{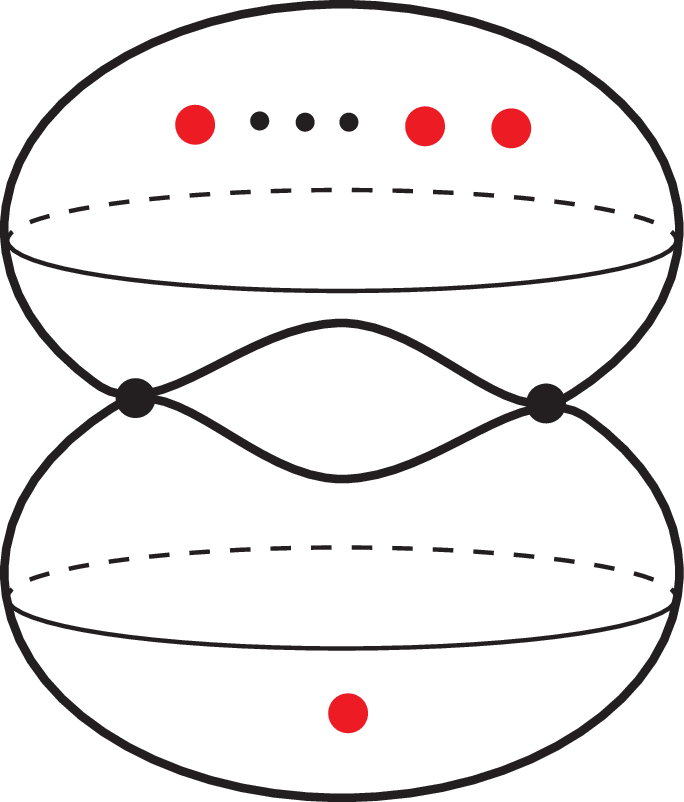}
	 \linethickness{1pt}
	\put(71, 87){$\Sigma_{\infty,1}$}  
	%\put(76,88) {\vector(-1,-2){5}}
	\put(71,0){$\Sigma_{\infty,2}$}
	%\put(76,8) {\vector(-1,2){5}}
	\put(-16,43){$\Sigma_{\infty}$}
	\end{overpic}
	\caption{Nodal Riemann surface $\Sigma_{\infty}$ with two nodes. Red points indicate marked points: $\Sigma_{\infty,1}$ contains $7-n$ marked points; $\Sigma_{\infty,2}$ contains one marked point.}
	\label{fig: nodal sphere}
\end{figure}

%\begin{theorem}
%Let $C$ be a symplectic surface in $(X,\omega)$ with $[C]=A$, where $A=3H-\sum_{i=1}^{n}E_{i}$. 
%Then, $C$ is symplectically isotopic to a complex curve in the same homology class. 
%\end{theorem}

\begin{proof}[Proof of Theorem \ref{thm: isotopy}]
Throughout this proof, we often denote an element $[\Sigma, j,u,\bm{x}]$ of $\overline{\M}_{1,8-n}(-K_n;J;\bm{z})$ just by $[u]$ for simplicity.

Take an $\omega_0$-tame almost complex structure $J$ on $M(n)$ making $S$ $J$-holomorphic and also choose $8-n$ points $\bm{z}$ on $S$. 
Notice that in general, $J \not \in \mathcal{J}^{*}_{\mathrm{reg}}(n;\bm{z})$. 
Thanks to the automatic regularity (Theorem \ref{thm: HLS}), the projection $\M_{1,8-n}(-K_n; \mathcal{J}_{\tau};\bm{z}) \rightarrow \mathcal{J}_{\tau}(M(n); \omega_0)$ is a submersion at a $J$-holomorphic curve $u \colon (\Sigma,j) \rightarrow M(n)$ parametrizing $S$. 
Hence one can slightly deform $([u],J)$ in the moduli space $\M_{1,8-n}(-K_n;\mathcal{J}_{\tau}; \bm{z})$ and obtain $([u'],J') \in \M_{1,8-n}(-K_n; \mathcal{J}_{\tau};\bm{z})$ with $J' \in \mathcal{J}^{*}_{\mathrm{reg}}(n;\bm{z})$ and $u'$ symplectically isotopic to $u$. 
%\note{Is $([\Sigma,j,u,\bm{x}],J) \in \M^{*}(-K_n; \bm{z}, \mathcal{J})$?}

The complex structure $J(n)$ may not be an element of $\mathcal{J}^{*}_{\mathrm{reg}}(n;\bm{z})$. 
%\red{For fixed $8-n$ distinct points $\bm{w}$ on $\Sigma$, consider the evaluation map 
%$$
%	\mathrm{ev}_{\bm{w}} \colon \M_{1,8-n}(A; J(n)) \rightarrow M(n) \times \cdots \times M(n) = (M(n))^{8-n}. 
%$$
%There exists a regular value $\bm{z}'$ of $\mathrm{ev}_{\bm{w}}$ sufficiently close to $\bm{z}$, which implies that $J(n) \in \mathcal{J}^{*}_{\mathrm{reg}}(n; \bm{z}')$. }
Choose distinct $8-n$ points $\bm{z}'=\{z'_1,\ldots,z'_{8-n}\} \subset M(n)$ sufficiently close to $\bm{z}$ for which $J(n) \in \mathcal{J}^{*}_{\mathrm{reg}}(n; \bm{z}')$.
A result of Boothby \cite{Boothby} yields a Hamiltonian isotopy $(\phi_{t})_{t \in [0,1]}$ of $(M(n), \omega_0)$ such that $\phi_{0}=\mathrm{Id}$ and $\phi_{1}(z'_i)=z_i$, 
and we have the regular almost complex structure 
$$
	(\phi_{1})_{*}J(n) \coloneqq d\phi_1 \circ J(n) \circ (d\phi_1)^{-1}\in \mathcal{J}^{*}_{\mathrm{reg}}(n; \bm{z}).
$$ 
Thus, once we have a symplectic isotopy $(S_t)$ between $J'$- and $(\phi_{1})_{*}J(n)$-holomorphic curves, the composition of this isotopy and $\phi_{t}^{-1}(S_{1})$ gives a symplectic isotopy between $S$ and a $J(n)$-holomorphic curve, i.e., a complex curve. 

Now we may assume that $S$ is a $J$-holomorphic, and $J$ and $J(n)$ are elements of $\mathcal{J}_{\textrm{reg}}^{*}(n; \bm{z})$. 
The moduli space along a path $h(t)$ in $\mathcal{J}_{\textrm{reg}}^{*}(n; \bm{z})$ connecting $J$ and $J(n)$,
$$\M_{1,8-n}(-K_n;h;\bm{z})=\bigcup_{t \in [0,1]} (\{t\} \times \M_{1,8-n}(-K_n;h(t); \bm{z}))$$ 
is a $3$-dimensional manifold. 
Moreover, it follows from \cite[Theorem~4.5.1]{Wen_Lecture} (see alternatively \cite[Theorem~2.15]{WenBook} for a statement of more straightforward application) together with the automatic transversality (Theorem \ref{thm: HLS}) that the natural projection $p \colon \M_{1,8-n}(-K_n; h; \bm{z}) \rightarrow [0,1]$ 
is a submersion. 

Let $I \subset [0,1]$ be the set of $\tau \in I$ such that for any $t \leq \tau $, a smooth $J_{t} \coloneqq h(t)$-holomorphic curve $u_{t}$ exists and is symplectically isotopic to $S$.
Since $p$ is a submersion, $I$ is an open subset of $[0,1]$. 
Hence all we have to show is that $I$ is closed in $[0,1]$. Let $(t_{\nu})$ be an increasing sequence of $I$ converging to $\tau$. 
For each $t_{\nu}$, choose a $J_{t_{\nu}}$-holomorphic curve $u_{t_\nu}$ that passes through $\bm{z}$ and is symplectically isotopic to $S$. 
By the Gromov compactness, we may assume that $[u_{t_\nu}]$ converges to a stable $J_{\tau}$-holomorphic curve $[u_{\tau}]$. 
Since $J_{\tau} \in \mathcal{J}^{*}_{\mathrm{reg}}(n;\bm{z})$, Lemma \ref{lem: N=1} implies that if $u_{\tau}$ is singular, i.e., an element of $\overline{\M}_{1,8-n}(-K_n; J_{\tau}; \bm{z}) \setminus \M_{1,8-n}(-K_n; J_{\tau}; \bm{z})$, then it is a nodal sphere with one node. 
Now Theorem \ref{thm: Sikorav} proves that the moduli space of such nodal curves contributes a $1$-dimensional stratum to $\overline{\M}_{1,8-n}(-K_n;h;\bm{z})$. 
Thus, its complement in $\overline{\M}(-K_n; h; \bm{z})$ is still connected near $[u_{\tau}]$. 
It turns out that a $J_{\tau}$-holomorphic smoothing of $u_{\tau}$ is symplectically isotopic to $u_{t_\nu}$ for $\nu \gg 0$, in particular to $S$. 
Therefore, $\tau \in I$, and $I$ is a closed subset of $[0,1]$. 
As $[0,1]$ is connected, we have $I=[0,1]$. 
This concludes the existence of a symplectic isotopy between $S$ and a $J(n)$-holomorphic curve.

An anti-canonical divisor of $M(n)$ is very ample for $0 \leq n \leq 6$ \cite[Theorem V.4.6]{Hart}.
Hence, the assertion for two symplectic submanifolds of dimension $2$ in the class $-K_n$ for $0 \leq n \leq 6$ immediately follows from the above discussion and the lemma below.  
\end{proof}

\begin{lemma}
Let $(M, \omega,J)$ be a K\"{a}hler manifold and $L \rightarrow M$ a very ample holomorphic line bundle. 
Then, any two smooth divisors $D$ and $D'$ in the complete linear system $|L|$ are symplectically isotopic as symplectic submanifolds with respect to $\omega$. 
\end{lemma}

\begin{proof}
Since $J$ is compatible with $\omega$, any smooth divisor on $(M, \omega,J)$ is a symplectic submanifold. 
Therefore, it suffices to construct a path of smooth divisors connecting $D$ and $D'$. 
Furthermore, since the holomorphic line bundle $L \rightarrow M$ is very ample, the K\"ahler manifold $M$ can be embedded into $\CP^{N}$ via the associated linear system $|L|$. 
Then, the divisors $D$ and $D'$ are realized as the hyperplane sections $M \cap H_D$ and $M \cap H_{D'}$, respectively, for some complex hyperplanes $H_D$ and $H_{D'}$ in $\CP^N$.
In general, Bertini's theorem (see \cite[page~137]{GH} for example) tells us that the intersection of $M$ and a generic complex hyperplane in $\CP^N$ gives a \textit{smooth} divisor in $|L|$. (Notice that the base locus of $|L|$ is empty by very ampleness.) 
It is a standard fact that a complex hyperplane in $\CP^N$ can be regarded as an element of the dual projective space $(\CP^N)^{*} \cong \CP^N$; see \cite[page~15]{GH}. 
In view of this, the space $\mathcal{H}$ of complex hyperplanes giving smooth hyperplane sections of $M \subset \CP^N$ agrees with the complement of the \textit{dual variety} $M^{*}$ in $(\CP^N)^*$, which is defined as the closure of the set of complex hyperplanes tangent to $M$ at some point (see \cite{Tel} for the details). 
The dual variety $M^*$ is a subvariety of $(\CP^N)^*$ of complex codimension at least $1$; hence, the space $\mathcal{H}$ is path-connected. 
A path from $H_D$ to $H_{D'}$ in $\mathcal{H}$ provides a one-parameter family of smooth ample divisors connecting $D$ to $D'$, which gives the desired symplectic isotopy. 
\end{proof}

\section{Mapping class group relation}\label{section: MCG}

\subsection{Del Pezzo surfaces of degree $6$}\label{section: del Pezzo}

A vital ingredient of the proof of Theorem \ref{thm: relation} is del Pezzo surfaces of degree $6$. 
In this subsection, we review how they appear as ample divisors of complex $3$-manifolds following \cite{Fujita} and then study the symplectic nature of those divisors. 

\subsubsection{Del Pezzo surfaces as ample divisors on complex $3$-folds}

%which can be realized as ample divisors of the following two complex $3$-manifolds: 
%\begin{enumerate}
%\item $\CP^1 \times \CP^1 \times \CP^1$; 
%\item A degree $(1,1)$-hypersurface of $\CP^2 \times \CP^2$.
%\end{enumerate}

Let $X_{1} \coloneqq  \CP^1 \times \CP^1 \times \CP^1$ and $M_1$ a smooth complex hypersurface of tri-degree $(1,1,1)$ in $X_{1}$. 
Also, let $X_{2}$ be a smooth complex hypersurface of bi-degree $(1,1)$ in $\CP^2 \times \CP^2$ and 
$M_{2}$ the complete intersection of $X_{2}$ and another smooth complex hypersurface of the same bi-degree in $\CP^2 \times \CP^2$.
A straightforward computation with the adjunction formula shows that $(K_{M_{j}})^2=6$ for $j=1,2$, and hence 
$M_{j}$ is a del Pezzo surface of degree $6$. 
For future use, we define the complex hypersurface $B_{j}$ in $M_{j}$ to be the intersection 
of $M_{j}$ and another generic smooth complex hypersurface of the same degree in $X_{j}$. 
%Since the line bundle $\O(1,1,1) \rightarrow X_{1}$ (resp. $\O(1,1) \rightarrow X_{2}$)
%is isomorphic to $\O(M_{1}) \rightarrow X_{1}$ (resp. $\O(M_{2}) \rightarrow X_{2}$), 
%With the adjunction formula again, 
By definition, the normal bundle $N_{M_{j}/X_{j}}$ of $M_{j}$ in $X_{j}$ is isomorphic to $-K_{M_j}$ and has $c_{1}(N_{M_{j}/X_{j}}) = \PD([B_{j}]) \in H^{2}(M_{j}; \Z)$, which implies that $B_{j}$ is an anti-canonical divisor of $M_{j}$. 
%\note{}
Moreover, we find that the canonical bundle $K_{B_j}$ is trivial, and $B_{j}$ is a genus $1$ curve in $M_{j}$ with 
self-intersection number $6$. 
Therefore, $M_j$ is a del Pezzo surface of degree $6$.

\subsubsection{Symplectic nature of $M_1$ and $M_2$}\label{section: symplectic del Pezzo}

We endow $M_j$ with a symplectic structure as follows: 
Let $\Omega_{j}$ be a symplectic form on $X_{j}$ induced by the Fubini--Study form on the product of complex projective spaces. 
As $M_j$ is a hyperplane section of $X_j$, we have $[\Omega_{j}]=\PD([M_{j}]) \in H^{2}(X_j;\Z)$. 
%Define $\omega_{j}$ to be the pull-back of $\Omega_{j}$ under the inclusion $Y_{j} \hookrightarrow X_{j}$. 
Since $M_{j}$ is a complex submanifold of the K\"ahler manifold $(X_{j}, \Omega_j)$, the restriction $\omega_j \coloneqq \Omega_{j}|_{M_{j}}$ serves as a symplectic form on $M_j$.
Note that $[\omega_{j}]=\PD([B_{j}])$.

\begin{lemma}\label{lem: del Pezzo symplectomorphism}
Let $(M_{1}, \omega_{1})$ and $(M_{2}, \omega_{2})$ be del Pezzo surfaces of degree $6$ with symplectic structures as above. 
%Also let $B_1$ and $B_2$ be anti-canonical divisors on $M_1$ and $M_2$, respectively, given above. 
Then, $(M_1, \omega_1)$ and $(M_2, \omega_2)$ are symplectomorphic.
%\begin{enumerate}
%\renewcommand{\labelenumi}{(\roman{enumi})}
%\renewcommand\theenumi\labelenumi
%\item\label{prop: M} $(M_1, \omega_1)$ and $(M_2, \omega_2)$ are symplectomorphic. 
%\item\label{prop: complement} The complements of some tubular neighborhoods of $B_{1}$ and $B_{2}$ in $M_{1}$ and $M_{2}$, respectively, are 
%symplectomorphic. 
%\end{enumerate}
\end{lemma}

\begin{proof}
%First, we prove \ref{prop: M}. 
Since a del Pezzo surface of degree $6$ is unique up to biholomorphism (see \cite[Remark 24.4.1]{Manin}), 
one can take a biholomorphism $\phi \colon M_1 \rightarrow M_2$. 
The symplectic forms $\omega_{1}$ and $\phi^{*}\omega_{2}$ are compatible 
with the complex structure on $M_1$, and so is $(1-t)\omega_1+t\phi^{*}\omega_2$ for $t \in [0,1]$. 
As $[\omega_1]$ and $[\phi^{*}\omega_{2}]$ are Poincar\'e dual to an anti-canonical divisor, they are cohomologous. 
Hence, Moser's stability gives a symplectomorphism $\psi \colon (M_1, \omega_1) \rightarrow (M_1, \phi^{*}\omega_2)$. 
Thus, $\phi \circ \psi \colon (M_{1}, \omega_{1}) \rightarrow (M_{2}, \omega_{2})$ is the desired symplectomorphism. 
%Next, we prove \ref{prop: complement}. 
\end{proof}

Next we will see that the complements of some tubular neighborhoods of $B_1$ and $B_2$ are symplectomorphic with (strictly) contactomorphic boundaries. 
To begin with, we briefly review a symplectic model of a tubular neighborhood of a symplectic submanifold. 

Let $(M, \omega)$ be a closed integral symplectic manifold and $B$ a codimension $2$ symplectic submanifold. 
Suppose that $[\omega]=\PD([B])$. 
Take a complex line bundle $p \colon N \rightarrow B$ with $c_1(M)=[\omega|_B]$, which is isomorphic to the normal bundle of $B$ in $M$. 
For a Hermitian metric $\| \cdot \|$ on $N$, we choose a Hermitian connection $\nabla$ with curvature $-2\pi i \omega|_B$. 
Let $\alpha^{\nabla}$ denote the associated transgression $1$-form $\alpha^{\nabla}$  on $N \subset B$ defined by 
$\alpha^{\nabla}|_{u}(u) =0$, $\alpha^{\nabla}|_{u}(iu)=1/2\pi$ for every $u \in N \setminus B$, and 
$\alpha^{\nabla}|_{H^{\nabla}}=0$, where $H^{\nabla}$ is the horizontal distribution of $\nabla$. 
Define a symplectic form $\omega^{\nabla}$ on $ \{ u \in N \mid \|u\| < 1 \}$ to be 
$$
	\omega^{\nabla} = p^{*}(\omega|_{B})+d(r^2\alpha^{\nabla}),
$$ 
where $r$ is the radial coordinate along the fibers induced by $\| \cdot \|$. 
As $d\alpha^{\nabla}=-p^{*}(\omega|_B)$, we have $\omega^{\nabla} = d((r^{2}-1)\alpha^{\nabla})$, which is exact outside the zero-section. 
Set 
$$
	\lambda^{\nabla} = (r^{2}-1)\alpha^{\nabla}.
$$ 
Weinstein's symplectic tubular neighborhood theorem tells us that for some $\delta>0$, there is a symplectic embedding $\rho \colon (N_{\delta}, \omega^{\nabla}) \rightarrow (M, \omega)$ that maps the zero-section to $B$, where $N_{\delta}=\{ u \in N \mid \|u\| \leq \delta\}$.

\begin{remark}\label{rem: principal}
We can canonically regard $N_{\delta}$ as $P \times_{S^{1}} \D(\delta)$, where $\varpi \colon P \rightarrow B$ is the unit circle bundle of $N$ with respect to $\|\cdot \|$, $\D(\delta) \coloneqq \{z \in \C \mid |z| \leq \delta \}$ and $S^1=\R/2\pi\Z$ acts on $P \times \D(\delta)$ by 
$$
	e^{i \theta} \cdot (x, z)=(x \cdot e^{-i\theta}, e^{i\theta}z)
$$
for $e^{i\theta} \in S^1$, $(x,z) \in P \times \D(\delta)$. 
Under the identification of $N_{\delta}$ with $P \times_{S^{1}} \D(\delta)$, the symplectic form $\omega^{\nabla}$ corresponds to $$-\varpi^{*}(\omega|_B)+d(r^2 \alpha),$$ 
and $\lambda^{\nabla}$ corresponds to $(r^{2}-1)\alpha$. 
Here we set $\alpha \coloneqq \alpha^{\nabla}|_{P}$, which is a connection $1$-form on $P$ with $d\alpha=-\varpi^{*}(\omega|_{B})$, and $r$ is the radial coordinate on $\D(\delta)$.  
\end{remark}

Let us return to the discussion of the del Pezzo surfaces $M_1$ and $M_2$. 
For each $j =1,2$, take a holomorphic line bundle $p_j \colon \O(B_{j}) \rightarrow M_{j}$ with a generic section $s_j$ satisfying $s_{j}^{-1}(0)=B_j$. 
Applying the above construction to the bundle $p_j|_{B_{j}} \colon \O(B_{j})|_{B_j} \rightarrow B_j$, we find a closed tubular neigborhood of $B_j$ in $M_j$ that is symplectomorphic to the disc bundle $(N_{\delta_j}, \omega^{\nabla_j})$ of radius $\delta_j$, where $\nabla_j$ is a connection on $\O(B_{j})$ with curvature $-2\pi i \omega|_{B_j}$. 
Write $$\rho_j \colon (N_{\delta_j}, \omega^{\nabla_j}) \rightarrow (M_j, \omega_j)$$ for this symplectic embedding. 
Since $-2\pi i \omega|_{B_j}$ is the curvature of $\nabla_{j}$, 
we have $\omega_j=\frac{1}{4\pi}dd^{\C}\log\|s_j(x)\|^2)$ on $M_j \setminus B_j$. 
In this holomorphic setting, according to \cite[Section 7]{Biran}, we may arrange the embedding $\rho_j$ so that the pull-back $\rho_j^{*} (\frac{1}{4\pi}d^{\C}\log\|s_j(x)\|^2)$ coincides with $\lambda^{\nabla_j}$.

\begin{proposition}\label{prop: del Pezzo symplectomorphism}
For $j=1,2$, let $(M_j, \omega_j)$ be the del Pezzo surface of degree $6$ with the symplectic structure $\omega_j$ and let $B_j$ be the anti-canonical divisor as above. 
Also, let $s_j$ be a holomorphic section of the holomorphic line bundle $\O(B_j)$ with $s_j^{-1}(0)=B_j$. 
Then, there exists a symplectomorphism $\Phi \colon (M_1, \omega_1) \rightarrow (M_2, \omega_2)$ such that the following holds: 
\begin{enumerate}
	\item $\Phi(B_1) = B_2$.
	\item For some closed tubular neighborhoods $\nu_{M_j}(B_j)$ of $B_j$, $\Phi(\nu_{M_1}(B_1))=\nu_{M_2}(B_2)$ and $\Phi^{*}\left(\left(\frac{1}{4\pi}d^{\C}\log\|s_2\|^2\right)|_{\nu_{M_2}(B_2) \setminus B_2}\right)=\left(\frac{1}{4\pi}d^{\C}\log\|s_1\|^2\right)|_{\nu_{M_1}(B_1) \setminus B_1}$. 
	In particular, 
		$$\Phi^{*}\left(\left. \left(\frac{1}{4\pi}d^{\C}\log\|s_2\|^2\right)\right|_{\del \nu_{M_2}(B_2)}\right)=\left. \left(\frac{1}{4\pi}d^{\C}\log\|s_1\|^2\right)\right|_{\del \nu_{M_1}(B_1)},$$ 
	that is, $\Phi|_{\del \nu_{M_1}(B_1)}$ gives rise to a strict contactomorphism. 
\end{enumerate}

\end{proposition}

The proposition will be shown by combining Lemma \ref{lem: del Pezzo symplectomorphism} with the following lemma. 

\begin{lemma}\label{lem: nbhd}
Let $(M, \omega)$ be a closed integral symplectic manifold and $B$ a symplectic submanifold with $[\omega]=\PD([B])$. 
Let $\lambda_0$ and $\lambda_1$ be $1$-forms defined on $M \setminus B$ satisfying $d\lambda_1=d\lambda_2=\omega$.
Suppose that for each $j=0,1$, the closed tubular neighborhood of $B$ given by a symplectic embedding $\rho_j \colon (N_{\delta_j}, \omega^{\nabla_j}) \rightarrow (M, \omega)$ satisfies  ${\rho_j}^{*}\lambda_j=\lambda^{\nabla_j}$, where $\nabla_j$ is a connection on the normal bundle $N$ of $B$ in $M$ with curvature $-2\pi i \omega|_B$ and $N_{\delta_j} \subset N$ is the disc bundle of radius $\delta_j$ with respect to a Hermitian metric $\|\cdot \|_{j}$. 
Then, there exists a symplectomorphism $\Psi \colon (M, \omega) \rightarrow (M, \omega)$ such that the following holds: 
\begin{enumerate}
\item $\Psi(B)=B$. 
\item $\Psi^{*}\lambda_2=\lambda_1$ near $B$.
\end{enumerate}
\end{lemma}

\begin{proof}
Shrinking $N_{\delta_0}$, we may assume that $N_{\delta_0} \subset N_{\delta_1}$ and $\mathrm{Im}(\rho_0) \subset \mathrm{Im}(\rho_1)$. 
Also assume for a while that there exists a symplectic embedding $\kappa \colon (N_{\delta_0}, \omega^{\nabla_0}) \rightarrow (N_{\delta_1}, \omega^{\nabla_1})$ that preserves the zero-section and satisfies $\kappa^{*}\lambda^{\nabla_1}=\lambda^{\nabla_0}$; we will show this existence later. 
Consider the isotopy of symplectic embeddings $\rho'_t \colon (N_{\delta_0}, \omega^{\nabla_0}) \rightarrow (M, \omega)$ defined by 
$$
	\rho'_t(u) \coloneqq \rho_1(t \kappa(u)+(1-t) (\rho_1^{-1} \circ \rho_0)(u))
$$
for $t \in [0,1]$. 
Note that $\rho'_0=\rho_0$, $\rho'_1=\rho_1 \circ \kappa$ and ${\rho'}_1^{*}\lambda_1=\lambda^{\nabla_0}$. 
According to \cite[Proposition 4]{Aur}, this isotopy yields a family of symplectomorphisms $\Psi_t \colon (M, \omega) \rightarrow (M, \omega)$ such that $\Psi_t(B)=B$ for every $t \in [0,1]$, $\Psi_0=\mathrm{Id}$ and $\Psi_t|_{\mathrm{Im}\rho_0}=\rho'_t \circ \rho^{-1}_0$. 
At $t=1$, we have 
$\Psi^{*}_{1}(\lambda_1) = (\rho'_1 \circ \rho^{-1}_0)^{*}(\lambda_1)=(\rho^{-1}_0)^{*}\lambda^{\nabla_0}=\lambda_0$ on $\mathrm{Im}\rho'_1$.
Thus, $\Psi_1$ is the symplectomorphism we want.  

To complete the proof, we shall prove the existence of $\kappa$. 
As in Remark \ref{rem: principal}, we may regard $(N_{\delta_j}, \omega^{\nabla_{j}})$ as $(P_{j} \times_{S^1} \D(\delta_j), -\varpi_j^{*}(\omega|_B)+d(r^2 \alpha_j))$, where $\varpi_j \colon P_j \rightarrow B$ is the unit circle bundle of $N$, with respect to $\|\cdot\|_j$, carrying the connection $1$-form $\alpha_j=\alpha^{\nabla_j}|_{P_j}$. 
Notice that the unit circle bundle $P_j$ depends on the metric $\|\cdot\|_j$, and hence $P_{1}$ and $P_{2}$ do not necessarily coincide as subsets of $N$. 
A standard argument of Gray's theorem (see e.g. \cite[Theorem 2.2.2]{Gei}) shows that there exists a strict contactomorphism between $(P_1, \alpha_1)$ and $(P_2, \alpha_2)$ that covers a symplectomorphism $(B, \omega|_B) \rightarrow (B, \omega|_B)$. 
This leads to a symplectic embedding 
$$\kappa' \colon (P_{1} \times_{S^1} \D(\delta_1), -\varpi_1^{*}(\omega|_B)+d(r_1^2 \alpha_1)) \rightarrow  (P_{2} \times_{S^1} \D(\delta_2), -\varpi_2^{*}(\omega|_B)+d(r_2^2 \alpha_2)).$$ 
In view of the identification of $N_{\delta_j}$ with $P_{j} \times_{S^1} \D(\delta_j)$, this embedding gives the claimed map $\kappa$. 
\end{proof}

\begin{proof}[Proof of Proposition \ref{prop: del Pezzo symplectomorphism}]
It follows from Lemma \ref{lem: del Pezzo symplectomorphism} that there is a symplectomorphism $\varphi' \colon (M_1, \omega_1) \rightarrow (M_2, \omega_2)$. 
By construction, the symplectic submanifold $\varphi'(B_1)$ of $(M_2, \omega_2)$ is homologous to $B_2$. 
%\note{$[Z_2]=\PD_{M_2}([\omega_2])=\varphi'_{*}\PD_{M_1}\varphi'^{*}[\omega_2]=\varphi'_{*} \PD_{M_1}[\omega_1]=\varphi'_{*}[Z_1]=[\varphi'(Z_1)]$.}
Hence, from Theorem \ref{thm: isotopy}, they are symplectically isotopic, and the isotopy yields a symplectomorphism, say $\varphi''$, between the pairs $(M_2, \varphi'(B_1))$ and $(M_2, B_2)$ by \cite[Proposition 4]{Aur}. 
The composition $\varphi''\circ \varphi' \colon (M_1, \omega_1) \rightarrow (M_2, \omega_2)$ is a symplectomorphism satisfying $(\varphi'' \circ \varphi') (B_1)=B_2$. 
Recall that $s_j \colon M_j \rightarrow \O(B_j)$ is a holomorphic section with $s_j^{-1}(0)=B_j$ for $j=1,2$ and let 
$$
	\lambda_j \coloneqq \frac{1}{4\pi}d^{\C} \log \|s_j\|^2, 
$$
defined on $M_j \setminus B_j$.
It turns out that Lemma \ref{lem: nbhd} is applicable to $\lambda_1$ and $(\varphi'' \circ \varphi')^{*} \lambda_2$ near $B_1$; we obtain a symplectomorphism $\Psi \colon (M_1, \omega_1) \rightarrow (M_1, \omega_1)$ satisfying $\Psi(B_1)=B_1$ and $\Psi^{*}((\varphi'' \circ \varphi')^{*} \lambda_2)=\lambda_1$ on a neighborhood of  $B_1$. 
Therefore, $\varphi'' \circ \varphi' \circ \Psi$ meets the requirements of the proposition. 
\end{proof}

%\begin{remark}
%Both of the complements $M_j \setminus B_j$ admit Weinstein structures defined by the functions $-\log \|s_j\|^2$ and their gradient vector fields with respect to the metric $\omega_j(\cdot, J_i \cdot)$, where $J_i$ are the complex structures on $M_j$. 
%The author does not know whether the symplectomorphism $(M_1 \setminus B_1, \omega_1|_{M_1 \setminus B_1}) \rightarrow (M_2 \setminus B_2, \omega_2|_{M_2 \setminus B_2})$ followed by Proposition \ref{prop: del Pezzo symplectomorphism} preserves the Weinstein structures or not. 
%\end{remark}

\subsection{Proof of Theorem \ref{thm: relation}}\label{section: proof of thm}

Let $f \colon X\setminus B \rightarrow \CP^1$ be a Lefschetz pencil on a $2n$-dimensional closed symplectic manifold $X$ such that the closure of a regular fiber is symplectomorphic to $M$.
Recall that the global monodromy of $f$ is isotopic to a fibered Dehn twist $\tau_{\del W}$ along the boundary of $W= M \setminus \nu_{M}(B)$, where $\nu_{M}(B)$ is a tubular neighborhood of $B$ in $M$. 
Moreover, $\tau_{\del W}$ factors into a product of Dehn twists. 

\begin{lemma}%[{cf. \cite[Lemma 2.17]{Bru}}]
\label{lem: number}
The fibered Dehn twist $\tau_{\del W}$ is Hamiltonian isotopic to a product of 
$k$ Dehn twists along Lagrangian spheres in $M \setminus B$, where 
$$
	k=(-1)^{n}(\chi(X)-2\chi(M)+\chi(B)).
$$
\end{lemma}

\begin{proof}
The given Lefschetz pencil $f$ induces a Lefschetz fibration $X \setminus \nu_X(M) \rightarrow \D$ with fibers diffeomorphic to $W = M \setminus \nu_M(B)$, where $\nu_X(M)$ denotes a tubular neighborhood of $M$ in $X$. 
This fibration decomposes its total space $X \setminus \nu_X(M)$ into $W \times \D$ and $k$ $n$-handles smoothly \cite{Kas}. 
Hence, we have 
$$
	\chi(X)-\chi(M) =\chi(X \setminus \nu(M)) = \chi(W\times \D)+(-1)^{n}k = \chi(M) - \chi(B)+(-1)^nk, 
$$
which concludes the lemma.
\end{proof}

\begin{proof}[Proof of Theorem \ref{thm: relation}]
Let $X_j, M_j$ and $B_j$, for $j=1,2$, be the manifolds as in Section \ref{section: del Pezzo}. 
Write $W_j$ for the complement of a tubular neighborhood of $B_j$ in $M_j$ satisfying $(W_1, \omega_1|_{W_1})$ and $(W_2, \omega_2|_{W_2})$ are symplectomorphic as in Proposition~\ref{prop: del Pezzo symplectomorphism}. 
Equip $W_1$ with a Weinstein structure by the function $-\frac{1}{4\pi}\log\|s_1\|^2$ and its gradient vector field with respect to the metric $\omega_1(\cdot, J_1 \cdot)$.
Note that although $W_2$ admits a Weinstein structure as well, to prove the theorem, we have no need of it. 

A $1$-dimensional linear system containing $M_j$ defines a Lefschetz pencil $X_j \setminus B_j \rightarrow \CP^1$. 
As we discussed in Section \ref{section: LP to LF}, one can truncate the domain of the map $X_j \setminus M_j \rightarrow \C$ ($j=1,2$) to obtain a Lefschetz fibration $\pi_j$ with regular fibers symplectomorphic to $W_j$.
The global monodromy of $\pi_j$ is a product of Dehn twists along $L_{j,1}, \ldots, L_{j,k_j}$, where $k_1=4$ and $k_2=6$. 
Indeed, according to Lemma \ref{lem: number}, we have
\begin{align*}
	k_{1} & =  (-1)^3((1+3+3+1)-2(1+4+1)+(1-2+1))  = 4,\\
	k_{2}  & =  (-1)^3((1+2+2+1)-2(1+4+1)+(1-2+1)) = 6.
\end{align*}
Theorem \ref{thm: boundary-interior} states that this product is isotopic to a fibered Dehn twist along $\del W_j$. 
The complements $W_1$ and $W_2$ are symplectomorphic, and their boundaries $\del W_1$ and $\del W_2$ carry contact structures which are strictly contactomorphic from Proposition~\ref{prop: del Pezzo symplectomorphism}. 
Thus, under the identification of $W_1$ with $W_2$, the two fibered Dehn twists $\tau_{\del W_1}$ and $\tau_{\del W_2}$ are symplectically isotopic in $\Symp_{c}(W_1, \omega_1|_{W_1})$. 
Therefore, we have 
\begin{align*}
	[\tau_{L_{1,1}} \circ \cdots \circ \tau_{L_{1,4}}]=[ \tau_{\del W_1}] = [\tau_{\del W_2}]=[\tau_{L_{2,1}} \circ \cdots \circ \tau_{L_{2,6}}]
\end{align*}
in $\pi_0(\Symp_{c}(W_1, \omega_1|_{W_1}) )$. This is the desired relation between the two products of Dehn twists. 
\end{proof}

%\begin{remark}
%\red{The fibered Dehn twist is not smoothly isotopic to the identity.} %in $\Diff_{+}(X)$.
%\end{remark}

%\begin{proposition}
%\red{The fibered Dehn twist vanishes in the mapping class group of $Y$.}
%\end{proposition}

%\red{
%\begin{theorem}
%The relation is a genuine symplectic phenomenon. 
%\end{theorem}
%}

%\subsection{Comparison between our relation and Keating's relation}

\subsection{Application of Theorem \ref{thm: relation}}\label{section: application}

We will construct closed contact $5$-manifolds having arbitrarily many finite Weinstein fillings, each with a distinct Euler characteristic. 
The construction uses open books and Lefschetz fibrations. 
The reader is referred to \cite[Section 7.3]{Gei} and \cite{vK} for open books and relation of them to contact structures and also referred to \cite{GP} for relation of Lefschetz fibrations to Weinstein domains. 

\begin{proposition}\label{prop: fillings}
Given a positive integer $n$, there exists a closed contact $5$-manifold $(Y_n, \xi_n)$ admitting at least $n$ Weinstein fillings each of which has a distinct Euler characteristic.
\end{proposition}

\begin{proof}
Let $W_1$ and $W_2$ be the Weinstein domains obtained above, and let $L_{1,1}, \ldots, L_{1,4}$ and $L_{2,1}, \ldots, L_{2,6}$ be Lagrangian spheres in $W_1$ and $W_2$, respectively, as in the proof of Theorem \ref{thm: relation}. 
Since $W_1$ and $W_2$ are symplectomorphic, we identify them and regard $L_{2,1}, \ldots, L_{2,6}$ as Lagrangian spheres in $W_1$. 
Consider the abstract open book whose page is $W_1$ and whose monodromy is $\tau_{\del W_1}^{n-1}$.  
Let $(Y_n, \xi_n)$ denote the contact manifold associated to this open book.  
As $[\tau_{\del W_1}]=[\tau_{L_{1,1}}] \circ \cdots \circ [\tau_{L_{1,4}}]=[\tau_{L_{2,1}}] \circ \cdots \circ [\tau_{L_{2,6}}] \in \pi_0(\Symp_{c}(W_1, \omega_1|_{W_1}))$ by Theorem \ref{thm: relation}, one can factor $[\tau_{\del W_1}^{n-1}]$ into a product of Dehn twists in $n$ ways: 
$$
	[\tau_{\del W_1}^{n-1}]=([\tau_{L_{1,1}}] \circ \cdots \circ [\tau_{L_{1,4}}])^{k} \circ ([\tau_{L_{2,1}}] \circ \cdots \circ [\tau_{L_{2,6}}] )^{n-1-k}
$$
for $k=0,\ldots,n-1$. 
Each factorization gives a Lefschetz fibration whose vanishing cycles coincides with Lagrangian spheres appearing in the factorization, and the total space of the fibration, say $E_k$, serves as a Weinstein filling of $(Y_n, \xi_n)$. 
The Euler characteristic of the filling $E_k$ is given by 
$$
	\chi(E_k)=\chi(W_1)-(4k+6(n-1-k))=\chi(W_1)-2(3n-3-k).
$$
Hence, we have $\chi(E_k) \neq \chi(E_{k'})$ if $k \neq k'$.  
\end{proof}

\appendix
\section{Alternative proof of Theorem \ref{thm: boundary-interior}}\label{appendix: boundary-interior}

In this appendix, we give an alternative proof of Theorem \ref{thm: boundary-interior} for complex projective manifolds. 
The precise statement to show is the following: 

\begin{theorem}[Theorem \ref{thm: boundary-interior} in the projective setting]\label{thm: appendix}
Let $X \subset \CP^{N}$ be a closed complex projective manifold with the K\"ahler form $\omega$ induced by $\omega_{\mathrm{FS}}$ on $\CP^{N}$. 
Also let $f \colon X \setminus B \rightarrow \CP^{1}$ be a Lefschetz pencil determined by hyperplane sections and $W$ the complement of a tubular neighborhood of $B$ in the closure of a regular fiber of $f$. 
Then, a fibered Dehn twist $\tau_{\del W}$ along $\del W$ is symplectically isotopic to a product of Dehn twists along Lagrangian spheres in $W$.
\end{theorem}

%\section{Lefschetz--Bott fibrations and Lefschetz pencils}

\subsection{Lefschetz--Bott fibrations}

First we briefly review Lefschetz--Bott fibrations (see \cite{Per} for a thorough treatment).

\begin{definition}
A symplectic form $\omega$ on an almost complex manifold $(M,J)$ is said to be 
\textit{normally K\"ahler} near an almost complex submanifold $N \subset M$ if 
there exists a tubular neighborhood $\mathcal{U}$ of $N$ in $M$, foliated by $J$-holomorphic normal discs $\{N_{x} \subset \mathcal{U}\}_{x \in N}$, such that $J$ is compatible with $\omega$ on $\mathcal{U}$; on each normal slice $N_x$, $J$ is integrable and $\omega$ is K\"ahler for $J$.
\end{definition}

Let $E$ be an even-dimensional open manifold equipped with a closed $2$-form $\omega$.

\begin{definition}\label{def: LB}
A smooth map $\pi: (E, \Omega) \rightarrow \C$ is called a \textit{Lefschetz--Bott fibration} if 
the set of critical points of $\pi$, denoted by $\Crit(\pi)$, is a submanifold of $E$ with finitely many connected components; 
there exists a compatible almost complex structure $J$ defined in a tubular neighborhood $\mathcal{U}$ of 
$\Crit(\pi)$ in $E$ and a positively oriented complex structure $j$ defined in a tubular neighborhood $\mathcal{V}$ of $\Critv(\pi)$ in $\C$, which satisfy the following conditions: 
\begin{enumerate}
\item The map $\pi|_{\mathcal{U}}: \mathcal{U} \rightarrow \mathcal{V}$ is $(J,j)$-holomorphic. 
\item $\Omega$ is nondegenerate on $\ker(D_{x}\pi)$ for any point $x \in E$. 
\item\label{condition: LB3} $\Omega$ is normally K\"ahler near the submanifold $\Crit(\pi)$. 
\item\label{condition: LB4} At any critical point $x \in \Crit(\pi)$, the normal complex Hessian $D^2_{x} \pi|_{N_{x}^{\otimes 2}}$ of $\pi$ is nondegenerate. 
\end{enumerate}
\end{definition}

Let us observe this definition. 
The condition (\ref{condition: LB3}), together with condition (\ref{condition: LB4}), shows that $\pi$ meets the conditions (\ref{cond: near critical points}) and  (\ref{cond: Hessian}) in the definition of Lefschetz fibration (Definition~\ref{def: Lefschetz}) when restricted to each normal disc $N_x$. 
Hence, on $\mathcal{U}$, the map $\pi$ is regarded as a family of Lefschetz fibrations, which is a complex analogue of the relationship between Morse and Morse--Bott functions.
In particular, $\pi$ is a Lefschetz fibration when $\dim_{\R} \Crit(\pi)=0$.

\begin{remark}\label{rem: compact fibers II}
We may impose the horizontal triviality on a Lefschetz--Bott fibration. 
If a Lefschetz--Bott fibration satisfies this condition, one can cut off cylindrical parts from the total space to obtain a map with compact fibers as in Remark \ref{rem: compact fibers}.  
\end{remark}

%We  recall some materials related to Lefschetz--Bott fibrations in short. 
Let $\pi:(E,\Omega) \rightarrow \C$ be a Lefschetz--Bott fibration on a $2n$-dimensional manifold $E$. 
For simplicity, we assume that each fiber of $\pi$ contains at most one connected component of $\Crit(\pi)$. 
Fix a critical value $z_1 \in \Critv(\pi)$ and write $C$ for the component of $\Crit(\pi)$ corresponding to $z_{1}$. 
Although the (real) dimension of $C$ can be between $0$ and $2n-4$ in general, this paper only deals with $C$ of dimension $0$ and $2n-4$. 
It suffices to examine the case $\dim_{\R} C=2n-4$ as we have already discussed Lefschetz fibrations. 
Fix a point $z_{0} \in \C \setminus \Critv(\pi)$ and 
choose a vanishing path $\gamma: [0,1] \rightarrow \C \setminus \Critv(\pi)$ for $z_1$ based at $z_0$. 
%Since $\pi$ is a symplectic fiber bundle away from the singular fibers, the symplectic connection determines the parallel transport along the restricted path $\gamma|_{[0,a]}$ for any $0 < a < 1$. 
%Define the \textit{vanishing thimble} $T_{\gamma}$ for $\gamma$ to be 
%$$
%	T_{\gamma} = \left( \bigcup_{0 \leq a <1}\{ x \in E_{\gamma(a)} \mid \lim_{b\rightarrow 1} h_{\gamma;a,b}(x) \in C \} \right) \cup C,  
%$$
%where $h_{\gamma;a,b}: E_{\gamma(a)} \rightarrow E_{\gamma(b)}$ is the parallel transport along $\gamma|_{[a,b]} \colon [a,b] \rightarrow S$. 
As in the case of Lefschetz fibrations, one can associate a \textit{coisotropic} submanifold $V_{\gamma}$ in $\pi^{-1}(z_{0})$ to $\gamma$, which is diffeomorphic to a circle bundle over $C$. 
We call $V_{\gamma}$ the \textit{vanishing cycle} of $\gamma$. 
Throughout this paper, $V_{\gamma}$ may be assumed to be the boundary of a Liouville (or Weinstein) domain with periodic Reeb orbits.
%In the case $\dim C=0$, $T_{\gamma}$ and $V_{\gamma}$ are diffeomorphic to an $n$-dimensional disc and an $(n-1)$-dimensional sphere, respectively; 
%in the case $\dim C=2n-4$, 
%Take a sufficiently small disc centered at $y_{1}$ with the orientation induced by $\C$ and connect $\gamma$ to the disc to obtain a loop $\ell$ based at $y_{0}$ 
Let $\ell$ be a loop based at $z_0$ obtained from $\gamma$ as in Figure \ref{fig: loop}. 
The monodromy of $\pi$ along $\ell$ agrees with a fibered Dehn twist along $V_{\gamma}$.

\subsection{Lefschetz--Bott fibrations on holomorphic line bundles}

Let $(M, J_{M}, \omega_{M})$ be a compact K\"ahler manifold of $\dim_{\C}M=n$ with $\omega_{M}$ integral.  
Suppose that $H$ is a smooth and reduced complex hypersurface in $M$ with $[\omega_{M}]=\PD([H])$. 
% \red{The condition of being reduced tends to be guaranteed by Bertini's theorem in this paper. Also this condition is required to express $H$ as the zero set of a holomorphic section.}
We write $p: \O(-H)\rightarrow M$ for the dual bundle of a holomorphic line bundle 
$\O(H) \rightarrow M$ corresponding to $H$.  
Let $J$ be an integrable almost complex structure on $\O(-H)$ for which $p$ is $(J,J_M)$-holomorphic, and fix a Hermitian metric on $\O(-H)$.
Choose a holomorphic section $s: M \rightarrow \O(H)$ with $s^{-1}(0)=H$. 
Note that $s$ is transverse to the zero-section since $H$ is reduced. 
%(see \cite[Proposition 2.3.18]{Huy} for its existence), 
We define the map $\pi: \O(-H) \rightarrow \C$ by 
$$
	\pi(x) \coloneqq \langle s(p(x)), x\rangle, 
$$
where $\langle \cdot , \cdot \rangle$ denotes the canonical coupling. 
We equip $\O(-H)$ with a canonical symplectic structure $\Omega= \Omega^{\nabla}$ determined by a Hermitian connection $\nabla$ with curvature $2\pi i \omega_{M}$ as in Section \ref{section: symplectic del Pezzo}. 
By construction, $J$ is $\Omega$-compatible. 

\begin{proposition}\label{prop: LBF}
The map $\pi: (\O(-H), \Omega) \rightarrow \C$ is a Lefschetz--Bott fibration. 
\end{proposition}

\begin{proof}
As $\pi$ is a $(J,j)$-holomorphic map for the standard complex structure $j$ on $\C$, 
the kernel of $D\pi$ at each point $x$ is a $J$-complex vector space. 
Thanks to the $\Omega$-compatibility of $J$, $\Omega$ is nondegenerate on $\ker(D_{x}\pi)$. 
Now all we have to do is examine the behavior of $\pi$ near the critical point set $\Crit(\pi)$.  

We first specify $\Crit(\pi)$. 
Take an open cover $\{ U_{\alpha}\}$ of $M$ such that 
$p|_{p^{-1}(U_{\alpha})}$ can be trivialized. 
For convenience, identifying $p^{-1}(U_{\alpha})$ with $U_{\alpha} \times \C$, we regard the restriction $s|_{U_{\alpha}}$ as a map from $U_\alpha$ to $U_{\alpha} \times \C$. 
Set $s_{\alpha} \coloneqq \mathrm{pr}_2 \circ s|_{U_{\alpha}}$, where $\mathrm{pr}_2 \colon U_{\alpha} \times \C \rightarrow \C$ is the projection to the second factor. 
On $p^{-1}(U_{\alpha})$, the function $\pi$ can be written as 
\begin{eqnarray*}
	\pi(b,z) =\langle (s \circ p)(b,z), (b,z)\rangle 
	 =  \langle (b,s_{\alpha}(b)), (b,z)\rangle 
	 =  s_{\alpha}(b)z  
\end{eqnarray*}
for $(b,z) \in U_{\alpha} \times \C$; this value is independent of the choice of coordinates. 
%The transition function of the dual bundle of a given bundle is the transpose inverse of the given bundle.
The Jacobi matrix of $\pi$ is given by 
$$
	\left( \frac{\del s_{\alpha}}{\del b_{1}}(b)z, \ldots, \frac{\del s_{\alpha}}{\del b_{n}}(b)z, s_{\alpha}(b)\right),
$$
where $(b_1, \ldots, b_n)$ are coordinates on $U_{\alpha}$.
Hence, $s_{\alpha}(b)$ must be $0$ for $(b,x) \in \Crit(\pi)$. 
Furthermore, we obtain $z=0$ by the transversality of the section $s$. 
Thus, 
$$
	\Crit(\pi)= s(M) \cap \{ \textrm{the zero-section}\} \eqqcolon H_0
$$ 
which is a complex submanifold of 
$\O(-H)$ biholomorphic to $H$. 

Now we shall construct a tubular neighborhood of $H_0$ that makes $\Omega$ normally K\"ahler. 
The complex Hessian matrix $D^2\pi$ at any critical point $(b,0) \in H_{0}$ is given by 
$$
	\renewcommand{\arraystretch}{1.3}\left.
   \begin{pmatrix} % or pmatrix or bmatrix or Bmatrix or ...
   
      \frac{\del^{2}s_{\alpha}}{\del b_{1} \del b_{1}}(b)z & \cdots & \frac{\del^{2}s_{\alpha}}{\del b_{1} \del b_{n}}(b)z  & \frac{\del s_{\alpha}}{\del b_{1}}(b) \\
      \vdots & \ddots & \vdots & \vdots \\
            \frac{\del^{2}s_{\alpha}}{\del b_{n} \del b_{1}}(b)z & \cdots & \frac{\del^{2}s_{\alpha}}{\del b_{n} \del b_{n}}(b)z  & \frac{\del s_{\alpha}}{\del b_{n}}(b) \\ 
       \frac{\del s_{\alpha}}{\del b_{1}}(b) & \cdots & \frac{\del s_{\alpha}}{\del b_{n}}(b) & 0
   \end{pmatrix} \right|_{(b,0)}
   =
      \begin{pmatrix} % or pmatrix or bmatrix or Bmatrix or ...
      0 & \cdots & 0  & \frac{\del s_{\alpha}}{\del b_{1}}(b) \\
      \vdots & \ddots & \vdots & \vdots \\
            0 & \cdots & 0  & \frac{\del s_{\alpha}}{\del b_{n}}(b) \\ 
       \frac{\del s_{\alpha}}{\del b_{1}}(b) & \cdots & \frac{\del s_{\alpha}}{\del b_{n}}(b) & 0
   \end{pmatrix}.
 $$ 
This shows that the complex matrix has complex rank $2$ and $\ker(D^2\pi)=TH_{0}^{\otimes 2}$. 
Consider the function $\|s\|^2 \colon M \rightarrow \R, x \mapsto \|s(x)\|^2$ and its gradient vector field $\nabla \|s\|^2$ with respect to the Riemannian metric $\omega(\cdot, J_M \cdot)$. 
We observe that $H \subset M$ is a Morse--Bott submanifold of this function, and the unstable manifold of $\nabla \|s\|^2$ at $x \in H$ is a holomorphic disc near $H$ since $\omega_{M}(\nabla \|s\|^2, J_M \nabla\|s\|^2)>0$ outside $H$ and a finite energy punctured holomorphic disc extends to a holomorphic disc. 
Hence, one can take a tubular neighborhood $U$ of $H$ in $M$ whose normal slices are holomorphic. %(see \cite[Section 7]{Biran} for example). 
Regarding the restriction $\O(-H)|_U$ of the bundle as a neighborhood of $H_0$ in $\O(-H)$, we obtain the desired tubular neighborhood $\mathcal{U}$ of $H_0$ in $\O(-H)$.
To check that $\mathcal{U}$ satisfies the conditions for $\Omega$ to be normally K\"ahler, we first find that $H_0$ is a complex submanifold of $(\mathcal{U}, J)$ since it is a complex submanifold of $\O(-H)$. 
By definition, $J$ is compatible with $\Omega$ on the whole $\O(-H)$, and it is the same on $\mathcal{U}$. 
In addition, by construction, the neighborhood $\mathcal{U}$ is foliated by complex $2$-dimensional discs $N_x$, each of which splits into the aforementioned normal slice in $U$ and a fiber of $\O(-H)$. 
Thus, each normal slice of $\mathcal{U}$ is $J$-holomorphic.
In view of the integrability of $J$ on $\O(-H)$, it is also integrable on each $N_x$, and $\Omega$ is K\"ahler for $J$.

The fact that $\ker (D^2\pi)=TH_{0}^{\otimes 2}$ proves that the normal Hessian of $\pi$ is nondegenerate at any point of $H_{0}$. 
This implies that $\pi$ meets the condition (\ref{condition: LB4}) in Definition \ref{def: LB}, which completes the proof. 
\end{proof}

\begin{remark}
A similar statement to the above theorem holds for closed integral symplectic manifolds $(M, \omega)$ having a codimension $2$ symplectic submanifold $\Sigma$ with $[\omega]=\PD([\Sigma])$. 
In fact, the author shows that a prequantization line bundle over such a closed integral symplectic manifold admits a Lefschetz--Bott fibration; see \cite{Oba}. 
\end{remark}

\subsection{Proof of Theorem \ref{thm: appendix}}

%\begin{theorem}[{Acu and Avdek \cite{AA2}, Gompf \cite[p.271]{Go} and Auroux \cite[p.6]{Au}}]\label{thm: boundary-interior}
%Let $(M, \omega)$ be a closed symplectic manifold admitting a Lefschetz pencil over $\CP^{1}$ and 
%$W$ the complement of the base locus in a regular fiber of the pencil. 
%Then, a fibered Dehn twist $\tau_{\del}$ along the boundary $\del W$ is Hamiltonian isotopic to the product of (right-handed) Dehn twists. 
%\end{theorem}

We are now in a position to prove Theorem \ref{thm: appendix}. 
The idea of the proof is to connect a Lefschetz fibration obtained from a Lefschetz pencil to a Lefschetz--Bott fibration on a certain space via one-parameter family of fibrations. 
In the following proof, given a complex projective manifold $M \subset \CP^N$, we denote the restriction of the hyperplane line bundle $\O_{\CP^{N}}(1) \rightarrow \CP^N$ to $M$ by $\O_{M}(1) \rightarrow M$.
%Also let $\omega_{M}$ denote the K\"ahler form on $M$ obtained as the restriction of the Fubini--Study form on $\CP^N$.

\begin{proof}[Proof of Theorem \ref{thm: appendix}]\
%Chow's theorem guarantees that 
Since a complex projective manifold is algebraic by Chow's theorem \cite[p.~167]{GH}, we may assume that the complex projective manifold $X$ is the zero set of homogeneous polynomials $f_{1}, \ldots, f_{k}$. 
Take generic hyperplane sections $s_{0}$ and $s_{1}$ that provides the given Lefschetz pencil $f \colon X \setminus B \rightarrow \CP^1$, i.e., 
$$
	f \colon X \setminus B \rightarrow \CP^1=\C \cup \{\infty\}, \quad f(z) = \frac{s_{1}(z)}{s_{0}(z)}, 
$$
where %$s_0$ and $s_1$ are regarded as homogeneous polynomials of degree $1$ in $\C[z_{0}, \ldots, z_{N}]$ and 
$B=s_{0}^{-1}(0) \cap s_{1}^{-1}(0) \cap X$.
Set $M \coloneqq s_0^{-1}(0) \cap X$ and suppose that it is smooth.
%Since $s_{0}$ and $s_{1}$ are chosen generically, 
%For any $z \in X \setminus B$, there exists a unique $\lambda \in \C \cup \{ \infty\} = \CP^1$ such that 
%$\lambda s_{0}(z)-s_{1}(z)=0$, which defines a Lefschetz pencil on $(X, \omega_{X})$, 
%$$
%	\CP^{N} \supset X \setminus B \rightarrow \CP^1, \ \ z \mapsto \lambda=\frac{s_{1}(z)}{s_{0}(z)}.
%$$
%Now for simplicity, we assume that $s_{0}=z_{0}$ and $s_{1}=z_{1}$, 
%where $(z_{0}: \ldots : z_{N})$ are homogeneous coordinates for $\CP^{N}$. 
%Hence we have 
%$$
%	H=M \cap \{z_{0}=0\}, \ \ B = H \cap \{ z_{1}=0\}. 
%$$
%Ye can always do this after coordinate change of $\CP^{N}$.
%Letting $Y=M \setminus H$, 
The complement $X \setminus M$ can be identified with the affine manifold $Y_{t}$ given by
$$
	Y_{t} = \{ x \in \C^{N+1} \mid s_{0}(x)=t, f_1(x)= \cdots = f_k(x)=0 \}
$$
for any $t \in \C \setminus \{0\}$. 
Indeed, the projection $\C^{N+1} \setminus \{0\} \rightarrow \CP^{N}$ yields a diffeomorphism between $Y_t$ and $X \setminus M$. 
By the definition of $\omega_{\textrm{FS}}$, $(Y_t, \rho_{\textrm{FS}}|_{Y_{t}})$ and $(X \setminus M, \omega_{\textrm{FS}}|_{X \setminus M})$ are symplectomorphic with $\rho_{\textrm{FS}}=-\frac{1}{4\pi}dd^{\C}\log\|x\|^2$. 
Notice that $M=\overline{f^{-1}(\infty)}$. 
Applying Proposition~\ref{prop: LP to LF} to the Lefschetz pencil $f \colon X \setminus B \rightarrow \CP^1$ gives a Lefschetz fibration on $X \setminus (M \cup \nu_{X}(B))$, where $\nu_{X}(B)$ is some tubular neighborhood of $B$ in $X$. 
This fibration agrees with the map 
$$
	\pi_{t} \colon (Y_{t}, \rho_{\textrm{FS}}|_{Y_t}) \rightarrow \C=\CP^{1} \setminus \{\infty\}, \quad x \mapsto s_{1}(x)/t 
$$ 
after truncating the fibers of $\pi_t$ and perturbing $\rho_{\textrm{FS}}|_{Y_t}$ as in Proposition~\ref{prop: LP to LF}. 
For simplicity, in what follows, we will regard $\pi_{t} \colon (Y_{t}, \rho_{\textrm{FS}}|_{Y_t}) \rightarrow \C$ itself as the Lefschetz fibration derived from the pencil $f$.

Next, we construct a Lefschetz--Bott fibration on the holomorphic line bundle $\O_{M}(-1) \coloneqq \O_{M}(1)^{*}$ by applying Proposition \ref{prop: LBF} to $p \colon \O_{M}(-1) \rightarrow M$. 
Observe that the bundle $\O_{M}(-1)$ is described as 
$$
	\O_{M}(-1) = \{ (z, x) \in \CP^{N} \times \C^{N+1} \mid s_0(z)=f_{1}(z)=\cdots=f_{k}(z)=0, x \in \ell_{z} \},
$$
where $\ell_{z}$ denotes the complex line in $\C^{N+1}$ spanned by $z$.
We may think of $s_1$ as a holomorphic section of $\O_{X}(1) \rightarrow X$ that restricts to one of $\O_{M}(1)\rightarrow M$ with $(s_1|_{M})^{-1}(0)=B$. 
Hence, Proposition \ref{prop: LBF} concludes that the map $\pi_{0} \colon (\O_{M}(-1), \Omega) \rightarrow \C$ defined by 
$$
	\pi_0(z,x) \coloneqq \langle s_{1}(z), (z,x) \rangle=s_1(x)
$$ 
is a Lefschetz--Bott fibration. 
Here $\Omega$ is a canonical symplectic form on the line bundle $\O_M(-1)$ as in Section \ref{section: symplectic del Pezzo}.

Now we relate this fibration $\pi_0$ to Lefschetz fibrations $\pi_t$ ($t \neq 0$). 
To do this, we shall first see $\O_{M}(-1)$ from a different viewpoint: 
Let $\tau \colon \O_{M}(-1) \rightarrow \C^{N+1}$ be the restriction of the projection $\CP^{N} \times \C^{N+1} \rightarrow \C^{N+1}$ to $\O_{M}(-1)$ and write $Y_{0}$ for the image of $\O_{M}(-1)$ under $\tau$. 
Notice that $Y_0$ is expressed as $$Y_0 =\{ x \in \C^{N+1} \mid s_0(x)=f_{1}(x)=\cdots=f_{k}(x)=0\}.$$
Since $M=s_0^{-1}(0) \cap X$ is assumed to be smooth, the projection $\C^{N+1} \setminus \{0\} \rightarrow \CP^N$ induces a smooth map $Y_0 \setminus \{0\} \rightarrow M$.
Thus, $Y_0$ has at most a unique singular point at the origin. 
%\note{Why unique? Because $\tau$ is isomorphic away from the origin and $\O_{M}(-1)$ is smooth.}
Hence, the bundle $\O_{M}(-1)$ can be regarded as a resolution of this singularity. 
%In this sense, we set  $\widetilde{X}=\O_{H}(-1)$. 
Next, let $\widetilde{Y}_{t}$ denote the strict transform of $Y_t$ under blowing up $\C^{N+1}$ at the origin. 
%\[
%   \xymatrix{
%    [0,1] \times \O(-1) \ar[r] \ar@{}[d]|{\bigcup} &  [0,1] \ar@{}[d]|{\bigcup} \times \C^{N+1} \\
%    \ar[r]^-{} \widetilde{\mathcal{Y}}  & \mathcal{Y}  }
%\] 
Set 
$$
	\widetilde{\mathcal{Y}} \coloneqq  \bigcup_{t \in [0,1]} \{t \} \times \widetilde{Y}_{t} \subset [0,1] \times \O_{\CP^N}(-1).
$$ 
By definition, $\widetilde{Y}_{0}$ coincides with $\O_{M}(-1)$, and the others $\widetilde{Y}_{t}$ are isomorphic to $Y_t $ since they do not contain $0 \in \C^{N+1}$. 
Define the map $\widetilde{\Pi} \colon \widetilde{\mathcal{Y}} \rightarrow [0,1] \times \C$ by 
$$
	\widetilde{\Pi}(t,z,x) \coloneqq (t, \langle s_1(z), (z,x) \rangle) = (t, s_1(x)), 
$$
which agrees with $\pi_{t}$ on every slice $\{t\} \times \widetilde{Y}_{t}$ up to a constant multiple. 

Take a K\"ahler form $\omega_{0}$ on the line bundle $\O_{\CP^N}(-1)$ such that it is a canonical symplectic form associated to a Hermitian connection near the zero-section, and far from here it equals the pull-back of $\rho_{\textrm{FS}}$ on $\C^{N+1} \setminus \{0\}$ by the blowing-up map. 
%\note{Is $\omega_{0}$ related to $\omega$ in the statement?}
We define a closed $2$-form on $[0,1] \times \O_{\CP^N}(-1)$ by $\mathrm{pr}_{2}^{*} \omega_{0}$, where $\mathrm{pr}_2 \colon [0,1] \times \O_{\CP^N}(-1) \rightarrow \O_{\CP^N}(-1)$ is the projection, and put 
$$
	\Omega_0=\mathrm{pr}_{2}^{*} \omega_{0}|_{\widetilde{\mathcal{Y}}}.
$$ 
Then, each regular fiber of $\widetilde{\Pi}$ is a symplectic manifold with respect to $\Omega_{0}$, which shows that 
$$
	\widetilde{\Pi}|_{\widetilde{\Pi}^{-1}(([0,1] \times \C) \setminus \Critv(\widetilde{\Pi}))} \colon \widetilde{\Pi}^{-1}(([0,1] \times \C) \setminus \Critv(\widetilde{\Pi})) \rightarrow ([0,1] \times \C) \setminus \Critv(\widetilde{\Pi})
$$ 
is a symplectic fibration. 
The symplectic connection defines parallel transport along a path in $([0,1] \times \C) \setminus \Critv(\widetilde{\Pi})$, especially the monodromy along a loop in the same space. %\note{Check this carefully.}
Perturbing the $2$-form $\Omega_{0}$ as in \cite[Lemma 9.3]{Keating}, we may assume that each fibration $\widetilde{\Pi}|_{\widetilde{Y}_{t}}$ is a Lefschetz(--Bott) fibration with horizontal triviality. 
As a result, in view of Remark \ref{rem: compact fibers} and \ref{rem: compact fibers II}, removing cylindrical parts from the fibers of $\widetilde{\mathcal{Y}}$, 
one can obtain a subspace $\mathcal{Y} \subset \widetilde{\mathcal{Y}}$ such that the restriction $\Pi \coloneqq \widetilde{\Pi}|_{\mathcal{Y}} \colon (\mathcal{Y}, \Omega_0|_{\mathcal{Y}})\rightarrow [0,1] \times \C$ has compact fibers symplectomorphic to the complement $(W, \omega_{\textrm{FS}}|_{W})$ of a tubular neighborhood of $B$ in $M$.  
The critical value set $\Critv(\Pi)$ of $\Pi$ consists of the family of finite sets $\Critv({\Pi}|_{\Pi^{-1}(\{t\} \times \C)})= \Critv(\widetilde{\Pi}|_{\widetilde{Y}_{t}})$ with $t \in [0,1]$. 
Thus, $([0,1] \times \C) \setminus \Critv(\Pi)$ is path-connected. 

Choose $x_{0} \in \C$ with $|x_{0}|$ sufficiently large and put $b_{0}=(0,x_{0}) \in [0,1] \times \C$. 
Identify $\Pi^{-1}(b_0)$ with $(W, \omega_{\textrm{FS}}|_W)$. 
%Note that $\widehat{W}$ has a cylindrical end. 
%Define the homomorphism $$\rho: \Pi_{1}(([0,1]\times \C) \setminus \Critv(\Pi), b_{0}) \rightarrow \Pi_{0}(\Symp_{c}(V, \omega)), \quad
%	\rho([\ell])=[\phi(\ell)], 
%$$ 
%where $\phi(\ell)$ is the monodromy of $\Pi_{\textrm{reg}}$ along a loop $\ell$ in $([0,1]\times \C) \setminus \Critv(\Pi)$ based at $b_{0}$.
Take a loop $\ell_{0}:[0,1] \rightarrow (\{ 0 \} \times \C) \setminus \Critv(\Pi) $ based at $b_{0}$ that encloses $( \{0 \} \times \C ) \cap \Critv(\Pi)$. 
As $\Pi|_{\Pi^{-1}(\{0\} \times \C)}$ is a Lefschetz--Bott fibration, the monodromy along $\ell_{0}$ is a fibered Dehn twist along $\del W$. 
Also, take a loop $\ell_{1} \colon [0,1] \rightarrow (\{ 1 \} \times \C) \setminus \Critv(\Pi) $ based at $(1,x_{0})$ surrounding all the points of $(\{1\} \times \C )\cap \Critv(\Pi)$. 
Since $\Pi|_{\Pi^{-1}(\{ 1 \} \times \C)}$ is a Lefschetz fibration, the monodromy of $\Pi$ along $\ell_{1}$ is a product of Dehn twists.
We connect $\ell_{1}$ with the path $\gamma \colon [0,1] \rightarrow [0,1] \times \C$, $t \mapsto (t,x_{0})$ and denote the resulting loop $\gamma \cdot \ell_1 \cdot \gamma^{-1}$ by $\ell'_1$ (see Figure \ref{fig: LBF_LF}). 
It turns out that the monodromy along $\ell'_1$  is isotopic to a product of Dehn twists, and $\ell'_1$ is homotopic to $\ell_0$ in $([0,1] \times \C) \setminus \Critv(\Pi)$. 
Therefore, the monodromies along $\ell_0$ and $\ell'_1$ are symplectically isotopic, which implies the desired mapping class group relation. 
\end{proof}

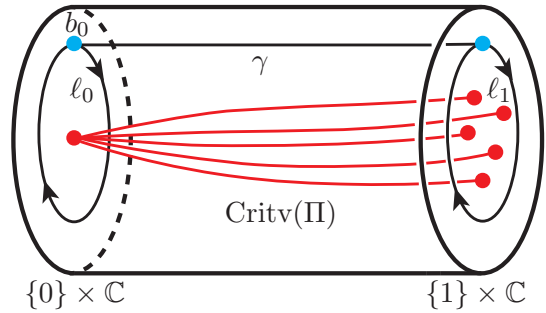
\begin{figure}[ht]
	\centering
	\begin{overpic}[width=200pt,clip]{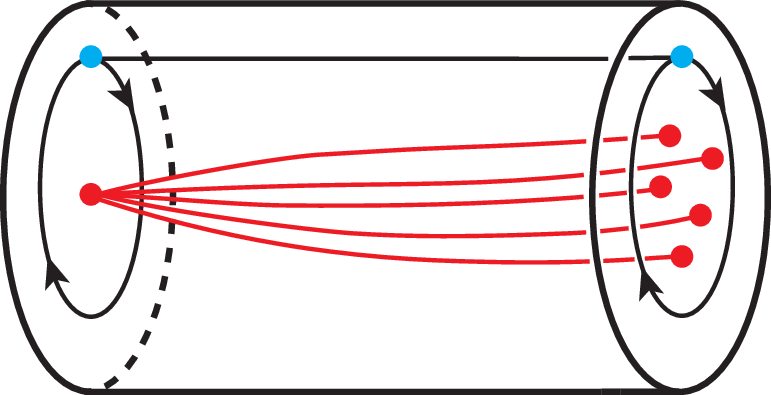}
	% \linethickness{3pt}  
 	\put(22,67){$\ell_{0}$}
	\put(178,67){$\ell_{1}$}
  	\put(5,-10){$\{0\} \times \C$}
	\put(155,-10){$\{1\} \times \C$}
	\put(20,92){$b_0$}
	\put(90,78){$\gamma$}
	\put(80,20){$\Critv(\Pi)$}
	\end{overpic}
	\vspace{10pt}
	\caption{Loops $\ell_0, \ell_1$ and path $\gamma$ in $I \times \C$. The red segments indicate the critical value set of $\Pi$. }
	\label{fig: LBF_LF}
\end{figure}

%next line adds the Bibliography to the contents page
%\addcontentsline{toc}{chapter}{Bibliography}
%\bibliographystyle{alpha}
%\bibliography{relation}

\end{document}